\documentclass[11pt,a4paper]{article}

\usepackage{latexsym}
\usepackage[T2A]{fontenc}
\usepackage[utf8x]{inputenc}
\usepackage[russian,english]{babel}

\usepackage[a4paper,top=3cm,bottom=4.5cm,left=2.9cm,right=2.9cm]{geometry}
\usepackage{graphicx}
\usepackage{longtable}
\usepackage{multirow}
\usepackage{amsfonts}
\usepackage{amsmath}
\usepackage{amsthm}
\usepackage{amssymb}
\usepackage{accents}
\usepackage{indentfirst}
\usepackage{enumerate}
\usepackage{float}
\usepackage{placeins}
\usepackage{mathtools}
\usepackage{csquotes}
\usepackage{setspace}
\usepackage{xifthen}
\usepackage{pdfpages}
\usepackage{transparent}
\usepackage[noadjust]{cite}
\usepackage{authblk}
\usepackage{arydshln}
\usepackage{subcaption}
\usepackage{bm}
\usepackage{centernot}
\usepackage{calc}

\usepackage{rotating}
\usepackage{pdflscape}
\usepackage{multicol}
\usepackage{xr-hyper}
\usepackage[colorlinks=true, allcolors=blue]{hyperref}

\makeatletter

\@addtoreset{equation}{section}

\newtheorem{theorem}{Theorem}[section]
\newtheorem{lemma}[theorem]{Lemma}
\newtheorem{corollary}[theorem]{Corollary}
\newtheorem{proposition}[theorem]{Proposition}

\theoremstyle{definition}
\newtheorem{example}[theorem]{Example}
\newtheorem{definition}[theorem]{Definition}
\newtheorem{notation}[theorem]{Notation}
\theoremstyle{remark}
\newtheorem{remark}[theorem]{Remark}

\newcommand\scpr[2]{\langle #1 \, \vert \, #2 \rangle}
\def\A{{\mathcal A}}
\def\Okubo{{\mathcal O}}
\def\Spin{{\mathcal Spin}}

\def\F{{\mathbb F}}

\def\N{{\mathbb N}}
\def\C{{\mathbb C}}

\def\chrs{{\rm char}\, }

\def\rank{{\rm rank}\,}
\def\tr{{\rm tr}\,}

\def\Lin{{\mathcal L}\,}
\def\dim{{\rm dim}\,}

\usepackage{xparse}

\makeatother

\newcommand{\multrow}[1]{\begin{tabular}{@{}c@{}} #1 \end{tabular}}

\providecommand{\keywords}[1]{\textbf{Keywords:} #1}
\providecommand{\msc}[1]{\textbf{MSC 2020:} #1}

\begin{document}
	
\title{On the lengths of descendingly flexible\\ and descendingly alternative algebras}
\author{A.~E.~Guterman$^a$ and S.~A.~Zhilina$^{b,c,d,}$\thanks{The work of the second author was supported by the Ministry of Science and Higher Education of the Russian Federation (Goszadaniye No. 075-00337-20-03, project No. 0714-2020-0005).}}
\date{\small \em
$^a$Department of Mathematics, Bar-Ilan University, Ramat-Gan, 5290002, Israel\\
$^b$Department of Mathematics and Mechanics, Lomonosov Moscow State\\ University, Moscow, 119991, Russia\\
$^c$Moscow Center for Fundamental and Applied Mathematics, Moscow, 119991, Russia\\
$^d$Moscow Institute of Physics and Technology, Dolgoprudny, 141701, Russia
}

\maketitle

\begin{abstract}
We introduce the classes of descendingly flexible and descendingly alternative algebras over an arbitrary field~$\F$. We suggest a new method based on the sequence of differences between the dimensions of the linear spans of words, which allows us to obtain upper bounds on the lengths of these algebras. We also present an example of an algebra of arbitrarily large dimension such that these bounds are achieved on it asymptotically.
\end{abstract}


\msc{15A03, 17A20, 17A30, 17D05.}

\let\thefootnote\relax\footnote{{\em Email addresses:} \texttt{alexander.guterman@biu.ac.il} (A.~E.~Guterman), \texttt{s.a.zhilina@gmail.com}\\ (S.~A.~Zhilina)}

\section{Introduction}

Let $\A$ be an arbitrary finite-dimensional algebra over a field $\F$, possibly non-unital and non-associative. Given a subset $S \subseteq \A$ and $k \in \N$, any product of $k$ elements of $S$ is called a {\em word of length $k$} in letters from~$S$. If the algebra is unital, the unity~$e$ is considered as a word of {\em zero length} in $S$. Otherwise, we assume that there are no words of zero length. 
 
Note that different arrangements of brackets provide different words of the same length. The set of all words in $S$ with length at most $k$ is denoted by $S^k$, $k \ge 0$. Then $k < m$ implies that $S^k \subseteq S^m$. The set of all words with letters from $S$ is denoted by $S^{\infty} = \bigcup \limits_{k=0}^\infty S^k$.

Recall that the linear span of $S$, denoted by $\Lin(S)$, is the set of all finite linear combinations of elements of $S$ with coefficients from~$\F$. We denote $\Lin_k(S) = \Lin(S^k)$, $k \geq 0$, and $\Lin_{\infty}(S) = \bigcup \limits_{k=0}^\infty \Lin_k(S)$. Clearly, for all $k$ we have $\Lin_k(S) \subseteq \Lin_{k+1}(S)$. Moreover, $\Lin_0(S) = \F$ if $\A$ is unital, and $\Lin_0(S) = \{ 0 \}$ otherwise. It can be easily seen that a set $S$ is generating for~$\A$ if and only if $\A = \Lin_{\infty}(S)$.

\begin{definition} \label{definition:system-length} 
The	{\em length of a set} $S$ is $l(S) = \min \{ k \: | \: \Lin_k(S) = \Lin_{\infty}(S) \}$. 
\end{definition}

\begin{definition} \label{definition:algebra-length}
The	{\em length of an algebra $\A$} is the maximum of lengths of its generating sets, i.e., $l(\A) = \max\limits_S \{l(S) \: | \: \Lin_{\infty}(S) = \A\}$. 
\end{definition}

The length function is an important numerical invariant which is useful for the study of finite-dimensional algebras. At present most papers in this area are devoted to the problem of length computation for the algebra $M_n(\F)$ of $n \times n$ matrices over a field $\F$, which was originally stated in~\cite{Paz}. Despite the fact that $M_n(\F)$ is well understood, this problem still remains open, though some bounds were obtained by several authors, see, e.g.,~\cite{Markova, Shitov}. One of the reasons for such a great interest in the length of $M_n(\F)$ is that it is closely related to the problem of classification of systems of complex inner product spaces and linear mappings between them, see \cite{Futorny}, which can be reduced to the problem of classifying complex matrices up to unitary similarity. By Specht's criterion, $A,B\in M_n({\mathbb C}) $ are unitarily similar if and only if $\tr(w(A,A^*))=\tr(w(B,B^*))$ for any word $w(x,y)$ in non-commuting variables $x,y$, see~\cite{Specht, Pearcy}. This criterion requires infinitely many tests but their number can be reduced by using an upper bound on the length of $M_n(\C)$, see~\cite{Futorny} for the details. It is the one of many various research directions where the length function can be applied.

Then the study of length function was extended to general associative algebras, and in~\cite{Pappacena} the upper bound for the length of an arbitrary finite-dimensional unital associative algebra was obtained. The case of associative algebras can be considered relatively simple in the following sense. If $\A$ is associative, then $\Lin_{k}(S) = \Lin_{k+1}(S)$ implies $\Lin_{k}(S)=\Lin_{\infty}(S)$. Therefore, for a unital algebra $\A$ we have $l(\A) \le \dim \A -1$, and for a non-unital algebra $\A$ we have $l(\A) \le \dim \A$. However, these inequalities do not hold for non-associative algebras, see~\cite[Example~2.8]{Guterman_upper-bounds}.

The notion of length was generalized to unital non-associative algebras in~\cite{Guterman_upper-bounds, Guterman_sequences} where the method of characteristic sequences was introduced. This method allows to prove a sharp upper bound for the lengths of unital non-associative algebras, cf.~\cite[Section~4]{Guterman_upper-bounds}, locally complex algebras, cf.~\cite[Section~6]{Guterman_upper-bounds}, and quadratic algebras, cf.~\cite[Section~3]{Guterman_quadratic}. We, however, prefer to use another extremely helpful tool for length computation. Below we introduce the sequence of differences for an arbitrary system $S$ which describes how fast the dimension of the linear span of words of length at most $k$ in $S$ is growing.

\begin{definition} \label{definition:sequence-differences}
Given a subset $S \subseteq \A$, the sequence of differences between the dimensions of the linear spans of words in $S$ is $D(S) = \{ d_k(S) \}_{k = 0}^{\infty} = \{ d_k \}_{k = 0}^{\infty}$, where
\begin{align*}
d_0 &{}= \dim \Lin_0(S) =
\begin{cases}
0, & \A \text{ is non-unital},\\
1, & \A \text{ is unital};
\end{cases}\\
d_k &{}= \dim \Lin_k(S) - \dim \Lin_{k-1}(S), \quad k \in \N.
\end{align*}
\end{definition}

Clearly, the definition of $d_0$ does not depend on a set $S$, and its value is a function of an algebra $\A$ itself. Hence we will assume throughout the paper that
$$
d_0 =
\begin{cases}
0, & \A \text{ is non-unital},\\
1, & \A \text{ is unital}.
\end{cases}
$$

The following proposition shows that all statements concerning lengths of algebras and their generating sets can be reformulated in terms of the sequence of differences. We denote the number of elements in $S$ by $|S|$, and the rank of $S$ as a linear system by $\rank(S)$.

\begin{proposition} \label{proposition:differences-length}
Let $\A$ be an arbitrary finite-dimensional algebra over a field $\F$, $S \subseteq \A$. Then
\begin{enumerate}[{\rm (1)}]
    \item $l(S) = \max \{ k \in \N_0 \; | \; d_k \neq 0 \}$, where we assume that $\max \varnothing = 0$;
    \item $S$ is generating for $\A$ if and only if $\sum \limits_{k = 0}^{\infty} d_k = \dim \A$;
    \item $d_1 =
    \begin{cases}
    \rank(S), & \A \text{ is non-unital},\\
    \rank(S \cup \{ e \}) - 1, & \A \text{ is unital}.
    \end{cases}$
\end{enumerate}
\end{proposition}

\begin{proof}
\leavevmode
\begin{enumerate}[{\rm (1)}]
    \item If $\A$ is non-unital and $S$ contains no nonzero elements, then $\Lin_{\infty}(S) = \Lin_0(S) = \{ 0 \}$ and $d_k = 0$ for all $k \in \mathbb{N}_0$, so $l(S) = 0 = \max \varnothing = \max \{ k \in \N_0 \; | \; d_k \neq 0 \}$. Otherwise, let us denote $m = l(S) \in \mathbb{N}_0$. Then $\{ 0 \} \neq \Lin_{\infty}(S) = \Lin_{m}(S)$, and $m$ is the smallest number with this property. Thus $d_{m} \geq 1$ and $d_k = 0$ for all $k > m$, so $m = \max \{ k \in \N_0 \; | \; d_k \neq 0 \}$.
    \item We have $\dim \Lin_{\infty}(S) = \dim \Lin_{m}(S) = \sum \limits_{k = 0}^{m} d_k = \sum \limits_{k = 0}^{\infty} d_k$, so $S$ is generating for $\A$ if and only if $\dim \A = \dim \Lin_{\infty}(S) = \sum \limits_{k = 0}^{\infty} d_k$.
    \item Note that $$
    \dim \Lin_1(S) = 
    \begin{cases}
    \rank(S), & \A \text{ is non-unital},\\
    \rank(S \cup \{ e \}), & \A \text{ is unital}.
    \end{cases}
    $$
    It remains to use the formulae for $d_0$ and $d_1$ from Definition~\ref{definition:sequence-differences}. \qedhere
\end{enumerate}
\end{proof}

Various identities satisfied by algebras can be very helpful in the study of their length functions. To generalize some of them, mixing and sliding algebras were introduced in~\cite{Guterman_slowly-growing}. It was shown that several important classes of algebras, including Leibniz, Novikov, and Zinbiel algebras, happen to be either mixing, or sliding, or both. Mixing and sliding algebras have slowly growing length, that is, their length does not exceed their dimension, see~\cite[Theorem~3.7]{Guterman_slowly-growing}.

The aim of our work is to study the lengths of several particular classes of composition algebras, such as standard composition algebras and Okubo algebras. This is an important step to the problem of length computation for more general classes of algebras, since, by~\cite[Theorem~2.9]{Elduque5},~\cite[Theorems~4.2 and~4.3]{ElduqueMyung1}, and~\cite[Theorem~3.2]{ElduqueMyung3},  these results would give us a complete description of lengths of symmetric composition algebras and finite-dimensional flexible composition algebras over an arbitrary field~$\F$.
In our previous paper~\cite{our_standard-composition-algebras} we have computed the lengths of standard composition algebras over an arbitrary field $\F$ with $\chrs \F \neq 2$. To do so, we have introduced the condition of descending flexibility which is satisfied by standard composition algebras, cf.~\cite[Lemma~4.5]{our_standard-composition-algebras}, and it has played a crucial role in our method of length computation. In the subsequent paper~\cite{GZh_Okubo} we will extend these results to standard composition algebras over~$\F$ with $\chrs \F = 2$ and compute the lengths of Okubo algebras over an arbitrary field~$\F$. Thus the problem of length computation for standard composition algebras and Okubo algebras will be solved completely. We will rely vastly on the main results of the current paper, namely, on the upper bounds for the lengths of descendingly flexible and descendingly alternative algebras.

In this paper we continue the study of the lengths of non-associative algebras. We introduce a new class of descendingly alternative algebras and re-introduce descendingly flexible algebras. Their definition is slightly altered with respect to the one which was given in~\cite{our_standard-composition-algebras}, so the class of descendingly flexible algebras becomes wider than it was before: by Proposition~\ref{proposition:sufficient-condition}, the algebras which satisfy~\cite[Definition~3.1]{our_standard-composition-algebras} are also descendingly flexible in the new sense. Since all the results obtained in~\cite[Section~3]{our_standard-composition-algebras} remain valid for the new definition of descendingly flexible algebras, we suggest that the modified definition is the one which should be used further. We establish several important properties of such algebras and show that the lengths of descendingly flexible and descendingly alternative algebras have at most logarithmic growth with respect to their dimensions.

Our paper is organized as follows. In Section~\ref{section:mixing-and-sliding} we recall the definitions of mixing and sliding algebras which were first given in~\cite{Guterman_slowly-growing}. By Lemma~\ref{lemma:general-recursive-length}, in these algebras it is sufficient to consider only those new words of length $l+1$ which can be obtained by multiplying a word of length $l$ by a single letter of a generating set, either on the left or on the right. This is a generalization of~\cite[Lemma~3.5]{our_standard-composition-algebras}, where this result was obtained for descendingly flexible algebras only.

In Section~\ref{section:almost-flexible-algebras} we define descendingly flexible and descendingly alternative algebras which are special cases of mixing algebras. We give some of their examples in Subsection~\ref{subsection:examples}. We show that descendingly flexible algebras are not necessarily flexible, descendingly alternative algebras may not be alternative, and the converse implications are not valid either. Besides, Example~\ref{example:classes-not-contained} demonstrates that the classes of descendingly flexible and descendingly alternative algebras are not contained in each other. In Example~\ref{example:lower-bound} we construct a $2^n$-dimensional algebra over an arbitrary field $\F$ with $\chrs \F = 2$, such that it is both descendingly flexible and descendingly alternative, and its length equals $n$. In Subsection~\ref{subseqtion:swappablity} we introduce the relation of swappability for the letters of a given word and show that it is symmetric and transitive.

In Sections~\ref{section:sequences-alternative} and~\ref{section:sequences-flexible} we obtain upper bounds on the lengths of descendingly alternative and descendingly flexible algebras, correspondingly. Their proofs follow the same line, but the case of descendingly flexible algebras is more technical. In Lemma~\ref{lemma:descendingly-alternative-classes} (Lemma~\ref{lemma:descendingly-flexible-bound}) we show that every new word in a descendingly alternative (descendingly flexible) algebra can be represented as a product of two (three) canonical factors. It then follows that all its letters are divided into at most two (three) classes of pairwise swappable letters. Then we show that, given an arbitrary set $S$, the sum of its sequence of differences depends at least exponentially on its length. Our main results are Theorems~\ref{theorem:descendingly-alternative-length} and~\ref{theorem:descendingly-flexible-length}. Their weaker but more convenient restatements are Corollaries~\ref{corollary:descendingly-alternative-length} and~\ref{corollary:descendingly-flexible-length} which show that the length of a descendingly alternative algebra does not exceed $\lceil \log_2(\dim \A - d_0) \rceil$, and the length of a descendingly flexible algebra of sufficiently large dimension is not greater than $\lceil \log_2(\dim \A - d_0) + \log_2(8/3) \rceil$.

\section{Irreducible words in mixing and sliding algebras} \label{section:mixing-and-sliding}

In this section we recall the definitions of mixing and sliding algebras which were introduced in~\cite{Guterman_slowly-growing}.

\begin{notation}
Given $x,y,z \in \A$, we denote
\begin{align*}
    Q_l(x,y,z) &{}= \{ x(zy), x(yz), y(xz), y(zx), xy, yx, xz, zx, yz, zy, x, y, z \},\\
    Q_r(x,y,z) &{}= \{ (xz)y, (zx)y, (yz)x, (zy)x, xy, yx, xz, zx, yz, zy, x, y, z \},\\
    P(x,y,z) &{}= Q_l(x,y,z) \cup Q_r(x,y,z).
\end{align*}
In other words, $Q_l(x,y,z)$ (or $Q_r(x,y,z)$) contains all monomials in $x, y, z$ of degree at most two whose letters are pairwise distinct, and only those monomials of degree three where the second (or the first) factor is $2$-fold and contains $z$.
\end{notation}

\begin{definition}
\leavevmode
\begin{itemize}
    \item An algebra $\A$ is called {\em left sliding} if $(xy)z \in \Lin_1(Q_l(x,y,z))$ for all $x,y,z \in \A$.
    \item An algebra $\A$ is called {\em right sliding} if $z(xy) \in \Lin_1(Q_r(x,y,z))$ for all $x,y,z \in \A$.
    \item An algebra $\A$ is called {\em sliding} if it is either left sliding or right sliding.
    \item An algebra $\A$ is called {\em mixing} if $(xy)z, z(xy) \in \Lin_1(P(x,y,z))$ for all $x,y,z \in \A$.
\end{itemize}
\end{definition}

\begin{remark}
Associative algebras are both mixing and left and right sliding.
\end{remark}

\begin{definition}
A class of algebras has {\em slowly growing length} if for any representative $\A$ of this class it holds that $l(\A) \leq \dim \A$.
\end{definition}

The next lemma shows that mixing and sliding algebras have the following property which holds trivially for associative algebras, namely, if the linear spans of words of two consequent lengths are equal, then the whole chain stabilizes afterwards.

\begin{lemma}[{\cite[Lemmas~3.5 and~3.6]{Guterman_slowly-growing}}] \label{lemma:full-stabilization}
Let $\A$ be mixing or sliding, and $S \subseteq \A$. If $\Lin_m(S) = \Lin_{m+1}(S)$ for some $m \ge 0$, then $\Lin_m(S) = \Lin_{m+1}(S) = \Lin_{m+2}(S) = \dots$.
\end{lemma}

\begin{corollary}[{\cite[Theorem~3.7]{Guterman_slowly-growing}}]
Mixing and sliding algebras have slowly growing length.
\end{corollary}

In terms of the sequence of differences, Lemma~\ref{lemma:full-stabilization} can be restated as follows.

\begin{proposition} \label{proposition:full-stabilization}
Let $\A$ be mixing or sliding, and $S \subseteq \A$. If $d_m = 0$ for some $m \in \N$, then $d_k = 0$ for all $k \geq m$.
\end{proposition}

We now use another approach and prove a stronger result than Lemma~\ref{lemma:full-stabilization}. In particular, the following lemma states that in mixing and sliding algebras it is sufficient to consider only those new words of length $l+1$ which can be obtained by multiplying a word of length~$l$ by a single letter of a generating set, either on the left or on the right.

\begin{lemma} \label{lemma:general-recursive-length}
Let $S$ be a subset of $\A$, $m \in \N$.
\begin{enumerate}[{\rm (1)}]
    \item If $\A$ is mixing, then
    $
    \Lin_{m+1}(S) = \Lin_m(S) \cdot S + S \cdot \Lin_m(S) + \Lin_m(S).
    $
    \item If $\A$ is left sliding, then
    $
    \Lin_{m+1}(S) = S \cdot \Lin_m(S) + \Lin_m(S).
    $
    \item If $\A$ is right sliding, then
    $
    \Lin_{m+1}(S) = \Lin_m(S) \cdot S + \Lin_m(S).
    $
\end{enumerate}
\end{lemma}

\begin{proof}
If $m = 1$, then the statement is true by the definition. Assume now that $m \geq 2$.
\begin{enumerate}[{\rm (1)}]
    \item Consider an arbitrary word $w \in S^{m+1}$. Then $w = xy$ for some $x \in S^{l+1}$ and $y \in S^{m-l}$, $l \in \{ 0, \dots, m-1 \}$. We prove that 
    $$
    w \in X_{m+1} := \Lin_m(S) \cdot S + S \cdot \Lin_m(S) + \Lin_m(S)
    $$
    by induction on $s(w) = \min \{ l(x), l(y) \}$:
    \begin{itemize}
        \item If $s(w) = 1$, that is, $l(x) = 1$ or $l(y) = 1$, then either $x \in S$ or $y \in S$, so $w \in (S^{m} \cdot S) \cup (S \cdot S^{m}) \subseteq X_{m+1}$.
        \item Let the statement be proved for all values less that $s(w) \geq 2$. Assume without loss of generality that $l(x) \leq l(y)$. Then $x = uv$ for some $u \in S^{k+1}$ and $v \in S^{l-k}$, $k \in \{ 0, \dots, l-1 \}$. By the definition of a mixing algebra, we have $w = xy = (uv)y \in \Lin_1(P(u,v,y))$. Any one-letter or two-letter word from $P(u,v,y)$ readily belongs to $S^m \subseteq X_{m+1}$. Any three-letter word $z \in P(u,v,y)$ is a product of two factors, exactly one of which contains $y$ as a strictly smaller subword. Thus we have either $s(z) = \min \{ l(u), l(v) + l(y) \} = l(u) < l(x) = s(w)$ or $s(z) = \min \{ l(v), l(u) + l(y) \} = l(v) < l(x) = s(w)$. Hence, by the inductive hypothesis, we obtain $z \in X_{m+1}$. Therefore, $w \in X_{m+1}$.
    \end{itemize}
    \item Consider an arbitrary word $w \in S^{m+1}$. Then $w = xy$ for some $x \in S^{l+1}$ and $y \in S^{m-l}$, $l \in \{ 0, \dots, m-1 \}$. We prove that 
    $$
    w \in Y_{m+1} := S \cdot \Lin_m(S) + \Lin_m(S)
    $$
    by induction on $s_l(w) = l(x)$:
    \begin{itemize}
        \item If $s_l(w) = l(x) = 1$, then $x \in S$, so $w \in S \cdot S^{m} \subseteq Y_{m+1}$.
        \item Let the statement be proved for all values less that $s_l(w) \geq 2$. We have $x = uv$ for some $u \in S^{k+1}$ and $v \in S^{l-k}$, $k \in \{ 0, \dots, l-1 \}$. By the definition of a left sliding algebra, it holds that $w = xy = (uv)y \in \Lin_1(Q_l(u,v,y))$. Any one-letter or two-letter word from $Q_l(u,v,y)$ belongs to $S^m \subseteq Y_{m+1}$. For any three-letter word $z \in Q_l(u,v,y)$ we have either $s_l(z) = l(u) < l(x) = s_l(w)$ or $s_l(z) = l(v) < l(x) = s_l(w)$. Hence, by the inductive hypothesis, we obtain $z \in Y_{m+1}$. Therefore, $w \in Y_{m+1}$.
    \end{itemize}
    \item This case is symmetric to the previous one. \qedhere
\end{enumerate}
\end{proof}

\begin{remark}
Lemma~\ref{lemma:full-stabilization} which was proved in~\cite{Guterman_slowly-growing} follows directly from Lemma~\ref{lemma:general-recursive-length}. Indeed, it is sufficient to show that $\Lin_m(S) = \Lin_{m+1}(S)$ implies $\Lin_{m+1}(S) = \Lin_{m+2}(S)$. This is clear if $m = 0$. If $m \geq 1$, then, by Lemma~\ref{lemma:general-recursive-length}, we have
\begin{align*}
\Lin_{m+2}(S) &{}= \Lin_{m+1}(S) \cdot S + S \cdot \Lin_{m+1}(S) + \Lin_{m+1}(S)\\
&{}= \Lin_m(S) \cdot S + S \cdot \Lin_m(S) + \Lin_m(S) = \Lin_{m+1}(S).
\end{align*}
Hence the statement follows.
\end{remark}

\begin{notation}
Let $a \in \A$. The mappings $L_a, R_a: \A \rightarrow \A$ are given by
\begin{align*}
L_a(x) &{}= ax,\\
R_a(x) &{}= xa
\end{align*}
for all $x \in \A$.
\end{notation}

\begin{notation} \label{notation:one-letter-at-a-time}
Let $S^{(m)} \subseteq S^m$ be the set constructed inductively by using equalities $S^{(1)} = S$ and $S^{(k+1)} = S^{(k)} \cdot S \cup S \cdot S^{(k)}$ for all $k \in \N$. In other words, 
$$
S^{(m)} = \{ X_1 X_2 \dots X_{m-1} s_m \; | \; X_i \in \{ L_{s_i}, R_{s_i} \} \text{ for } i = 1, \dots, m - 1, \text{ and } s_1, \dots, s_m \in S \}.
$$
\end{notation}

\begin{corollary} \label{corollary:new-elements}
Let $\A$ be mixing, $S \subseteq \A$. Then for any $m \in \N$ we have $\Lin_m(S) = \Lin_{m-1}(S) + \Lin(S^{(m)})$.
\end{corollary}

\begin{proof}
Follows immediately from Lemma~\ref{lemma:general-recursive-length}.
\end{proof}

\section{Descendingly flexible and descendingly alternative algebras} \label{section:almost-flexible-algebras}

\subsection{Definition and sufficient conditions}

Some of the non-associative algebras satisfy certain restrictions that are close to the associativity of their elements, and this makes their study easier. Probably the most popular among such restrictions are flexibility and alternativity.

\begin{definition} \label{definition:flexible-alternative}
\leavevmode
\begin{itemize}
\item An algebra $\A$ is called {\em flexible} if $(ab)a = a(ba)$ for all $a,b \in \A$.  
\item An algebra $\A$ is called {\em alternative} if $a(ab) = (aa)b$ and $(ba)a = b(aa)$ for all $a,b \in \A$.  
\end{itemize}
\end{definition}

It is well known that any alternative algebra is also flexible, see~\cite[p.~154, Exercise~2.1.1]{McCrimmon}, but the converse is not true, e.g., Cayley--Dickson algebras of dimension at least $16$ are flexible but not alternative, see~\cite[p.~436]{Schafer}.

Inspired by Definition~\ref{definition:flexible-alternative}, we introduce the notions of descendingly flexible and descendingly alternative algebras which form two important subclasses in the class of mixing algebras.

Consider $a, b, c \in \A$. We denote by $\Lin'_2(a,b,c)$ the linear span of all words of length at most two in $a,b,c$, except for the words $aa$, $bb$ and $cc$, that is,
$$
\Lin'_2(a,b,c) =
\begin{cases}
\phantom{e,{}}\Lin(a,b,c,ab,ba,cb,bc,ac,ca), & \A \text{ is non-unital},\\
\Lin(e,a,b,c,ab,ba,cb,bc,ac,ca), & \A \text{ is unital}.
\end{cases}
$$
Note also that
$$
\Lin_1(a,b,aa,ab,ba) = 
\begin{cases}
\phantom{e,{}}\Lin(a,b,aa,ab,ba), & \A \text{ is non-unital},\\
\Lin(e,a,b,aa,ab,ba), & \A \text{ is unital}.
\end{cases}
$$

\begin{definition}
\leavevmode
\begin{itemize}
\item We say that an algebra $\A$~is {\em descendingly flexible} if for all $a, b, c \in \A$ it holds that
\begin{align}
    (ab)a, \;\; a(ba) &{}\in \Lin_1(a,b,aa,ab,ba), \label{equation:general-flexibility}\\
    (ab)c + (cb)a &{}\in \Lin'_2(a,b,c), \label{equation:general-flexibility-1}\\
    a(bc) + c(ba) &{}\in \Lin'_2(a,b,c). \label{equation:general-flexibility-2}
\end{align}

\item We say that an algebra $\A$~is {\em descendingly alternative} if for all $a, b, c \in \A$ it holds that
\begin{align}
    (ba)a, \;\; a(ab) &{}\in \Lin_1(a,b,aa,ab,ba), \label{equation:general-alternativity}\\
    (ab)c + (ac)b &{}\in \Lin'_2(a,b,c), \label{equation:general-alternativity-1}\\
    a(bc) + b(ac) &{}\in \Lin'_2(a,b,c). \label{equation:general-alternativity-2}
\end{align}
\end{itemize}
\end{definition}

Clearly, if $\chrs \F \neq 2$, then Eqs.~\eqref{equation:general-flexibility-1} and~\eqref{equation:general-flexibility-2} imply Eq.~\eqref{equation:general-flexibility}, and Eqs.~\eqref{equation:general-alternativity-1} and~\eqref{equation:general-alternativity-2} imply Eq.~\eqref{equation:general-alternativity}. Therefore, conditions~\eqref{equation:general-flexibility} and~\eqref{equation:general-alternativity} are needed only in the case when $\chrs \F = 2$.

The following proposition provides a convenient sufficient condition for an algebra to be descendingly flexible or descendingly alternative. It is a stronger form of Eqs.~\eqref{equation:general-flexibility} and~\eqref{equation:general-alternativity}.

\begin{proposition} \label{proposition:sufficient-condition}
\leavevmode
\begin{enumerate}[{\rm (1)}]
    \item Assume that for all $a, b \in \A$ we have $(ab)a, \, a(ba) \in \Lin_1(a,b,aa,ab,ba)$, and the coefficient at $aa$ depends only on $b$. In other words, $(ab)a$ and $a(ba)$ can be represented as
    \begin{equation} \label{equation:general-flexibility-3}
    f_1(a,b)a + f_2(a,b)b + g(b)aa + f_3(a,b)ab + f_4(a,b)ba + f_5(a,b)e
    \end{equation}
    for some functions $f_j: \A \times \A \to \F$, $j = 1, \dots, 5$, and $g: \A \to \F$. If $\A$ is non-unital, then we assume that $f_5(a,b) = 0$ for all $a,b \in \A$. Then $\A$ is descendingly flexible.
    \item If for all $a, b \in \A$ it holds that $(ba)a, \, a(ab) \in \Lin_1(a,b,aa,ab,ba)$, and the coefficient at $aa$ depends only on $b$, then $\A$ is descendingly alternative.
\end{enumerate}
\end{proposition}

\begin{proof}
To prove (1), we linearize Eq.~\eqref{equation:general-flexibility-3} by considering $((a+c)b)(a+c) - (ab)a - (cb)c$ and $(a+c)(b(a+c)) - a(ba) - c(bc)$. The terms with $aa$ and $cc$ vanish, as the coefficient $g(b)$ at $aa$ depends only on $b$. The proof of (2) is completely similar.
\end{proof}

\subsection{Examples and main properties} \label{subsection:examples}

We now give several examples of descendingly flexible and descendingly alternative algebras.

\begin{example}
\leavevmode
\begin{enumerate}[{\rm (1)}]
    \item If a nonunital algebra $\A$ is descendingly flexible (alternative), then its unital hull $\F \oplus \A$ is also descendingly flexible (alternative).
    \item If $\A$ is nilpotent of index at most $3$, that is, $\A^3 = 0$, then $\A$ is both descendingly flexible and descendingly alternative.
    \item Let $\A$ be a standard composition algebra over an arbitrary field $\F$, see~\cite[p.~378]{ElduquePerez3} for its definition. If $\chrs \F \neq 2$, then, by~\cite[Lemma~4.5]{our_standard-composition-algebras}, $\A$ is descendingly flexible. Moreover, the proof can be extended to the case when $\chrs \F$ is arbitrary by using the same arguments.
    \item Consider an Okubo algebra $\Okubo$ over an arbitrary field $\F$, as defined in~\cite{Elduque1}. Then $\Okubo$ is descendingly flexible, since it satisfies $(ab)a = a(ba) = n(a)b$ for all $a,b \in \Okubo$, see~\cite[p.~284, pp.~286--287]{Elduque1} and~\cite[p.~1199]{ElduqueMyung1}.
    \item The algebra of spin factors $\Spin_n = \F 1 \oplus \F^n$ with multiplication given by
    $$
    (\alpha + v) (\beta + w) = (\alpha \beta + \scpr{v}{w}) + (\alpha w + \beta v),
    $$
    where $\scpr{\cdot}{\cdot}$ is the standard inner product on $\F^n$, is both descendingly flexible and descendingly alternative. Indeed, since $\Spin_n$ is commutative, one can verify that for $a = \alpha + v$ and $b = \beta + w$ we have
    $$
    (ab)a = a(ba) = (ba)a = a(ab) = (\scpr{v}{w} - \alpha \beta) a + \beta aa + \alpha ab.
    $$
\end{enumerate}
\end{example}

\begin{remark} \label{remark:not-flexible}
Descendingly flexible algebras are not necessarily flexible. For example, if $\chrs \F \neq 2$, then standard composition algebras of types~II and~III over $\F$ are descendingly flexible, cf.~\cite[Lemma~4.5]{our_standard-composition-algebras}, however, they are not flexible. One can verify that this is also an example of descendingly alternative algebras which are not alternative.

Conversely, flexible (alternative) algebras need not be descendingly flexible (alternative). Indeed, any associative algebra is automatically both flexible and alternative. Let $\A = M_n(\F)$ be the algebra of $(n \times n)$-matrices over an arbitrary field $\F$ with $n \geq 4$, and $\{ E_{ij} \; | \; i, j = 1, \dots, n \}$ be the set of matrix units. Set $a = E_{12}$, $b = E_{23}$, and $c = E_{34}$. Then $abc + cba = abc + acb = abc + bac = E_{14} \notin \Lin'_2(a,b,c)$. Hence Eqs.~\eqref{equation:general-flexibility-1}--\eqref{equation:general-flexibility-2} and Eqs.~\eqref{equation:general-alternativity-1}--\eqref{equation:general-alternativity-2} are not satisfied.
\end{remark}

\begin{proposition} \label{proposition:descending-are-mixing}
Descendingly flexible and descendingly alternative algebras are mixing. Hence they have slowly growing length, and one can apply Lemmas~\ref{lemma:full-stabilization} and~\ref{lemma:general-recursive-length} and Corollary~\ref{corollary:new-elements} to them.
\end{proposition}

\begin{proof}
Follows immediately from the definition of mixing algebras, Eqs.~\eqref{equation:general-flexibility-1}--\eqref{equation:general-flexibility-2} and Eqs.~\eqref{equation:general-alternativity-1}--\eqref{equation:general-alternativity-2}.
\end{proof}

The following example shows that the classes of descendingly flexible and descendingly alternative algebras are not contained in each other.

\begin{example} \label{example:classes-not-contained}
Let $\mathbb{F}$ be an arbitrary field, $\A_{\rm flex}$ be an algebra over $\mathbb{F}$ with the basis $e_1, \dots, e_5$ and multiplication table~\ref{table:multiplication table_1}, $\A_{\rm alt}$ be an algebra over $\mathbb{F}$ with the basis $f_1, \dots, f_5$ and multiplication table~\ref{table:multiplication table_2}. Then $\A_{\rm flex}$ is descendingly flexible but not descendingly alternative, and $\A_{\rm alt}$ is descendingly alternative but not descendingly flexible.

Indeed, $e_1 (e_1 e_2) = e_1 e_3 = e_4 \notin \Lin(e_1, e_2, e_1 e_1, e_1 e_2, e_2 e_1)$ and $f_1 (f_2 f_1) = f_1 f_3 = f_4 \notin \Lin(f_1, f_2, f_1 f_1, f_1 f_2, f_2 f_1)$, so $\A_{\rm flex}$ is not descendingly alternative, and $\A_{\rm alt}$ is not descendingly flexible. We now show that $a(ba) = (ab)a = 0$ for all $a, b \in \A_{\rm flex}$ and $c(cd) = (dc)c = 0$ for all $c, d \in \A_{\rm alt}$. Then $\A_{\rm flex}$ is descendingly flexible and $\A_{\rm alt}$ is descendingly alternative by Proposition~\ref{proposition:sufficient-condition}.

Clearly, the coefficients at $e_1$ and $e_2$ in $ab$ and the coefficients at $f_1$ and $f_2$ in $dc$ are zero, so $(ab)a = 0$ and $(dc)c = 0$. Hence it remains to show that $a(ba) = 0$ and $c(cd) = 0$. By linearity, it is sufficient to consider $b \in \{ e_1, e_2 \}$ and $d \in \{ f_1, f_3, f_5 \}$. Besides, since $e_4 x = x e_4 = 0$ for any $x \in \A_{\rm flex}$, we can take $a \in \Lin(e_1, e_2, e_3, e_5)$, and since we multiply by~$c$ only on the left, we can take $c \in \Lin(f_1, f_2)$. Let $a = k_1 e_1 + k_2 e_2 + k_3 e_3 + k_5 e_5$, $c = l_1 f_1 + l_2 f_2$. Then
\begin{align*}
e_1 a &{}= e_1 (k_1 e_1 + k_2 e_2 + k_3 e_3 + k_5 e_5) = k_2 e_3 + k_3 e_4 + k_1 e_5,\\
a (e_1 a) &{}= (k_1 e_1 + k_2 e_2 + k_3 e_3 + k_5 e_5) (k_2 e_3 + k_3 e_4 + k_1 e_5) = k_1 k_2 e_4 - k_1 k_2 e_4 = 0,\\
a (e_2 a) &{}= a (e_2 (k_1 e_1 + k_2 e_2 + k_3 e_3 + k_5 e_5)) = a (-k_5 e_4) = 0,\\
c (c f_1) &{}= (l_1 f_1 + l_2 f_2) ((l_1 f_1 + l_2 f_2) f_1) = (l_1 f_1 + l_2 f_2) (l_2 f_3 + l_1 f_5) = l_1 l_2 f_4 - l_1 l_2 f_4 = 0,\\
c (c f_3) &{}= c ((l_1 f_1 + l_2 f_2) f_3) = c (l_1 f_4) = 0,\\
c (c f_5) &{}= c ((l_1 f_1 + l_2 f_2) f_5) = c (-l_2 f_4) = 0,
\end{align*}
as required.
\end{example}

\begin{table}[H]
\centering
\begin{minipage}{0.5\textwidth}
\centering
$
\begin{array}{|c|ccccc|}
\hline
\times&e_1&e_2&e_3&e_4&e_5\\\hline
e_1&e_5&e_3&e_4&0&0\\
e_2&0&0&0&0&-e_4\\
e_3&0&0&0&0&0\\
e_4&0&0&0&0&0\\
e_5&0&0&0&0&0\\\hline
\end{array}
$
\caption{\label{table:multiplication table_1} Multiplication table of $\A_{\rm flex}$.}
\end{minipage}%
\begin{minipage}{0.5\textwidth}
\centering
$
\begin{array}{|c|ccccc|}
\hline
\times&f_1&f_2&f_3&f_4&f_5\\\hline
f_1&f_5&0&f_4&0&0\\
f_2&f_3&0&0&0&-f_4\\
f_3&0&0&0&0&0\\
f_4&0&0&0&0&0\\
f_5&0&0&0&0&0\\\hline
\end{array}
$
\caption{\label{table:multiplication table_2} Multiplication table of $\A_{\rm alt}$.}
\end{minipage}
\end{table}

To sum up, the classes of flexible, alternative, descendingly flexible and descendingly alternative algebras are related as follows:

\begin{table}[H]
\centering
\begin{tabular}{ccc}
\fbox{\multrow{Alternative}} & \multrow{$\Rightarrow$\\ $\centernot\Leftarrow$} & \fbox{\multrow{Flexible}}\\
$\centernot\Downarrow \; \centernot\Uparrow$ & & $\centernot\Downarrow \; \centernot\Uparrow$ \\[7pt]
\fbox{\multrow{Descendingly\\ alternative}} & \multrow{$\centernot\Rightarrow$\\ $\centernot\Leftarrow$} & \fbox{\multrow{Descendingly\\ flexible}}
\end{tabular}
\caption{Relations between (descendingly) flexible\\ and (descendingly) alternative algebras.}
\end{table}

\subsection{An example of a logarithmic bound on the length}

The following example shows that the upper bound on the lengths of descendingly flexible and descendingly alternative algebras has at least logarithmic growth with respect to their dimensions. 

\begin{example} \label{example:lower-bound}
Assume that $\chrs \mathbb{F} = 2$. For any $n \in \mathbb{N}$ consider an algebra~$\A$ over~$\mathbb{F}$ with the basis $\{ e_x \; | \; x \in \mathbb{Z}_2^n \}$, so $\dim \A = 2^n$. We define multiplication on $\A$ by the formula $e_x e_y = e_{x+y}$. Clearly, $\A$ is commutative, associative, and unital, with $e_0$ being its unit element.

We now prove that $\A$ is descendingly alternative, i.e., conditions~\eqref{equation:general-alternativity}--\eqref{equation:general-alternativity-2} hold. Due to commutativity of $\A$, it will follow automatically that $\A$ is descendingly flexible. To prove~\eqref{equation:general-alternativity}, note that for any $a = \sum_{x \in \mathbb{Z}_2^n} a_x e_x$ and $b \in \A$ we have $(ba)a = a(ab) = (aa)b = (\sum_{x \in \mathbb{Z}_2^n} a_x^2) b \in \Lin(b)$. Conditions~\eqref{equation:general-alternativity-1} and~\eqref{equation:general-alternativity-2} follow from the fact that $(ab)c + (ac)b = a(bc) + b(ac) = 2abc = 0$.

If $\{ f_1, \dots, f_n \}$ is an arbitrary basis in $\mathbb{Z}_2^n$, then $S = \{ e_{f_1}, \dots, e_{f_n} \}$ is a generating system for $\A$ which clearly has length $n$. Therefore, $l(\A) \geq n = \log_2 (\dim \A)$.
\end{example}

In Sections~\ref{section:sequences-alternative} and~\ref{section:sequences-flexible} we prove that the upper bound on the lengths of these algebras is actually logarithmic. Its exact value is still unknown, but we will see that for descendingly alternative algebras it does not exceed $\lceil \log_2(\dim \A - d_0) \rceil$ (see Corollary~\ref{corollary:descendingly-alternative-length}), and for descendingly flexible algebras of sufficiently large dimension it is not greater than $\lceil \log_2(\dim \A - d_0) + \log_2(8/3) \rceil$ (see Corollary~\ref{corollary:descendingly-flexible-length}).

\begin{remark}
We now compute the exact value of $l(\A)$ in Example~\ref{example:lower-bound}. Assume that $l(\A) \geq n + 1$. Then, by Theorem~\ref{theorem:descendingly-alternative-length} which will be proved later, we have 
$$
2^n - 1 = \dim \A - d_0 \geq 2^n + (n + 1) - 2 \geq 2^n,
$$
a contradiction. Hence $l(\A) \leq n$, so $l(\A) = n$.
\end{remark}

\subsection{Swappability relation} \label{subseqtion:swappablity}

Together with the sequence of differences which was defined in the introduction, one of the main tools to be used in the next two sections is the partition of all letters of a given word into several classes of pairwise swappable letters. To introduce the swappability relation, we first need the definitions of equivalence and formal equivalence for words in letters from $S$.

\begin{definition}
The {\em length of a word} $x \in S^{\infty}$ is the number of letters from $S$ in its notation, i.e., the smallest value $m \in \mathbb{N}_0$ such that $x \in S^m$. We denote it by $l(x)$.
\end{definition}

\begin{definition}
Let $x,y \in \pm S^{\infty}$ with $l(x) = l(y) = m \geq 1$. We say that $x$ and $y$ are {\em equivalent} if $x - y \in \Lin_{m-1}(S)$. We denote $x \sim y$.
\end{definition}

\begin{proposition} \label{proposition:equivalence:properties}
\leavevmode
\begin{enumerate}[{\rm (1)}]
    \item The relation $\sim$ is an equivalence relation.
    \item If $l(x) = l(y) = m \geq 1$, $x \sim y$ and $y \in \Lin_{m-1}(S)$, then $x \in \Lin_{m-1}(S)$.
    \item If $x \sim y$ and $z \in \pm S^{\infty}$, then $xz \sim yz$ and $zx \sim zy$.
\end{enumerate}
\end{proposition}

\begin{proof}
\leavevmode
\begin{enumerate}[{\rm (1)}]
    \item It is clear that $\sim$ is reflexive and symmetric. Besides, if $x \sim y$ and $y \sim z$ with $l(x) = l(y) = l(z) = m$, then $x - y, \: y - z \in \Lin_{m-1}(S)$ implies that $x - z = (x - y) + (y - z) \in \Lin_{m-1}(S)$, so $x \sim z$.
    \item Follows immediately from the the fact that $x - y \in \Lin_{m-1}(S)$.
    \item Let $l(x) = l(y) = m \geq 1$ and $l(z) = k$. Then $l(xz) = l(yz) = l(zx) = l(zy) = m + k$. Since $x - y \in \Lin_{m-1}(S)$, we have $xz - yz = (x-y)z \in \Lin_{m+k-1}(S)$ and $zx - zy = z(x-y) \in \Lin_{m+k-1}(S)$, so $xz \sim yz$ and $zx \sim zy$. \qedhere
\end{enumerate}
\end{proof}

\begin{remark}
Definition of equivalence can be extended to the case when $x, y$ are not just words of length $m$, but (formal) linear combinations of elements from $S^m$, and both~$x$ and~$y$ have nonzero summands which correspond to some words of length~$m$ in letters from~$S$. In this case Proposition~\ref{proposition:equivalence:properties} remains valid. For simplicity of notation, we will use this modified definition in the proofs of Lemmas~\ref{lemma:descendingly-flexible-length} and~\ref{lemma:pulling-one-letter-subwords} without further notice.
\end{remark}

\begin{proposition} \label{proposition:descendingly-flexible}
Assume that $a, b, c \in S^{\infty}$ with $l(a), \, l(b), \, l(c) \geq 1$.
\begin{enumerate}[\rm (1)]
\item If $\A$ is descendingly flexible, then
\begin{align}
    (ab)c &{}\sim -(cb)a, \label{equation:general-equivalence-1}\\
    a(bc) &{}\sim -c(ba). \label{equation:general-equivalence-2}
\end{align}
\item If $\A$ is descendingly alternative, then
\begin{align}
    (ab)c &{}\sim -(ac)b, \label{equation:general-equivalence-3}\\
    a(bc) &{}\sim -b(ac). \label{equation:general-equivalence-4}
\end{align}
\end{enumerate}
\end{proposition}

\begin{proof}
Let us denote $n = l(a) + l(b) + l(c)$. Then $l((ab)c) = l(a(bc)) = n$, and the statement follows immediately from Eqs.~\eqref{equation:general-flexibility-1}--\eqref{equation:general-flexibility-2} and Eqs.~\eqref{equation:general-alternativity-1}--\eqref{equation:general-alternativity-2}, since the lengths of the basis elements of $\Lin'_2(a,b,c)$ are less than $n$.
\end{proof}

Proposition~\ref{proposition:descendingly-flexible} allows us to introduce another equivalence relation for words in letters from~$S$ which is stronger than usual equivalence.

\begin{definition} \label{definition:formal-equivalence}
Let $S \subseteq \A$ and $x, y \in \pm S^{\infty}$. We say that $x$ and $y$ are {\em formally equivalent} if there is a chain of elementary equivalences between $x$ and $y$. Elementary equivalences are defined inductively:
\begin{itemize}
    \item If an equivalence has the form described in Eqs.~\eqref{equation:general-equivalence-1}--\eqref{equation:general-equivalence-2} (if~$\A$ is descendingly flexible) or Eqs.~\eqref{equation:general-equivalence-3}--\eqref{equation:general-equivalence-4} (if~$\A$ is descendingly alternative), then it is elementary.
    \item If $u \sim v$ is an elementary equvalence, then for any $z \in S^{\infty}$ the equivalences $uz \sim vz$ and $zu \sim zv$ are also elementary. 
\end{itemize}
\end{definition}

\begin{definition} \label{definition:swappable}
Let $S \subseteq \A$, $x \in S^{\infty}$, and $s, t \in S$ are contained in $x$ (denoted by $s, t \in x$). We say that $s$ and $t$ {\em are swappable} (denoted by $s \leftrightarrow t$) if one of the following conditions holds:
\begin{enumerate}[{\rm (1)}]
    \item $\A$ is descendingly flexible, and $x$ is formally equivalent to some word either of the form $\dots R_s L_t \dots$ or of the form $\dots L_s R_t \dots$;
    \item $\A$ is descendingly alternative, and $x$ is formally equivalent to some word either of the form $\dots R_s R_t \dots$ or of the form $\dots L_s L_t \dots$.
\end{enumerate}
\end{definition}

\begin{proposition} \label{proposition:alternating-letters}
\leavevmode
\begin{enumerate}[{\rm (1)}]
    \item The relation $\leftrightarrow$ is symmetric and transitive.
    \item If $x \in S^m$, $s, t \in x$, and $s \leftrightarrow t$, then $s = t$ implies $x \in \Lin_{m-1}(S)$.
    \item If $x \in S^m$, and for some $n > |S|$ the letters $s_1, \dots, s_n \in x$ are pairwise swappable, then $x \in \Lin_{m-1}(S)$.
\end{enumerate}
\end{proposition}

\begin{proof}
\leavevmode
\begin{enumerate}[{\rm (1)}]
    \item Symmetricity follows from Eqs.~\eqref{equation:general-equivalence-1}--\eqref{equation:general-equivalence-4}. Assume now that $s \leftrightarrow t \leftrightarrow u$, and 
    $$
    x = \dots s \dots t \dots u \dots \sim \dots X_s Y_t \dots \sim \dots Z_t W_u \dots,
    $$
    where $X, Y, Z, W \in \{ L, R \}$. Then we have
    \begin{align*}
    x &{}\sim \dots X_s Y_t \dots \sim - \dots X_t Y_s \dots \\
    &{}\sim - \dots t \dots s \dots u \dots \sim - \dots Z_s W_u \dots,
    \end{align*}
    so $s \leftrightarrow u$.
    \item Follows immediately from Proposition~\ref{proposition:equivalence:properties}(2) together with Eqs.~\eqref{equation:general-flexibility} and~\eqref{equation:general-alternativity}.
    \item Since $n > |S|$, there exist $i, j \in \{ 1, \dots, n \}$ such that $i \neq j$ but $s_i = s_j$. Then the statement follows from~(2). \qedhere
\end{enumerate}
\end{proof}

\section{Lengths of descendingly alternative algebras} \label{section:sequences-alternative}

In this section we assume that $\A$ is a descendingly alternative algebra over an arbitrary field~$\F$. Recall that the set $S^{(m)} \subseteq S^m$ was introduced in Notation~\ref{notation:one-letter-at-a-time}.

\begin{lemma} \label{lemma:descendingly-alternative-classes}
Let $S \subseteq \A$ and $w \in S^{(m)}$ for some $m \geq 2$. Then
\begin{enumerate}[{\rm (1)}]
    \item There exists a word $xy = (((x_1 x_2) \dots) x_k) \cdot (y_{m-k} (\dots (y_2 y_1)))$, where $1 \leq k \leq m-1$, such that $w \sim \pm xy$.
    \item There are at most two classes of pairwise swappable letters in $xy$, namely, $y_1, x_2, \dots, x_k$ and $x_1, y_2, \dots, y_{m-k}$.
\end{enumerate}
\end{lemma}

\begin{proof}
\leavevmode
\begin{enumerate}[{\rm (1)}]
    \item We prove the statement by induction on $m$. Clearly, the induction base is true. In order to prove the induction step, we note that if $w = xy \in S^m$ for $x = ((x_1 x_2) \dots) x_k$ and $y = y_{m-k} (\dots (y_2 y_1))$, then for any $s \in S$ we have, by Eqs.~\eqref{equation:general-equivalence-3} and~\eqref{equation:general-equivalence-4}, $ws = (xy)s \sim -(xs)y$ and $sw = s(xy) \sim -x(sy)$, which are again of the desired form.
    \item Note that 
    \begin{align*}
    xy &{}= (((x_1 x_2) \dots) x_k) \cdot (y_{m-k} (\dots (y_2 y_1))) \\
    &{}\sim -y_{m-k} ((((x_1 x_2) \dots) x_k) \cdot (y_{m-k-1} (\dots (y_2 y_1)))) \\
    &{}\sim \dots \sim \pm y_{m-k} (\dots (y_{m-k-1} ((((x_1 x_2) \dots) x_k) \cdot y_1))),
    \end{align*}
    so $y_1, x_2, \dots, x_k$ are pairwise swappable. Similarly, $x_1, y_2, \dots, y_{m-k}$ are pairwise swappable. \qedhere
\end{enumerate}
\end{proof}

\begin{proposition}[{\cite[Lemma~2.11, Corollary~2.12]{Guterman_upper-bounds}}] \label{proposition:equal-linear-spans}
Let $S, S'$ be two subsets of $\A$ such that $\Lin_1(S) = \Lin_1(S')$. Then $\Lin_k(S) = \Lin_k(S')$ for all $k \in \N$, so $\Lin_{\infty}(S) = \Lin_{\infty}(S')$. Hence $S$ is generating for $\A$ if and only if $S'$ is generating for~$\A$, and $l(S) = l(S')$.
\end{proposition}

\begin{corollary} \label{corollary:descendingly-alternative-bound}
Let $S \subseteq \A$. Then $\Lin_{2d_1(S) + 1}(S) = \Lin_{2d_1(S)}(S)$, so $l(S) \leq 2d_1(S)$.
\end{corollary}

\begin{proof}
Recall that $d_1 = d_1(S) = \dim \Lin_1(S) - \dim \Lin_0(S)$. It follows that $\Lin_1(S) = \Lin_0(S) + \Lin(S')$ for some $S' \subseteq S$ such that $|S'| = d_1$. Then $\Lin_1(S) = \Lin_1(S')$, so, by Proposition~\ref{proposition:equal-linear-spans}, $\Lin_k(S) = \Lin_k(S')$ for all $k \in \N$. Thus we may assume that $|S| = d_1(S) = d_1$.

Consider an arbitrary word $xy \in S^{2d_1 + 1}$ such that
$$
x = ((x_1 x_2) \dots) x_k \ \mbox{ and } \ y = y_{2d_1+1-k} (\dots (y_2 y_1))
$$
where $x_1, \dots, x_k, y_1, \dots, y_{2d_1 + 1 - k} \in S$ and $1 \leq k \leq 2d_1$. By Lemma~\ref{lemma:descendingly-alternative-classes}(2), $y_1, x_2, \dots, x_k$ are pairwise swappable ($k$ letters in total) and $x_1, y_2, \dots, y_{2d_1+1-k}$ are pairwise swappable ($2d_1 + 1 - k$ letters in total). Since either $k \geq d_1 + 1$ or $2d_1 + 1 - k \geq d_1 + 1$, Proposition~\ref{proposition:alternating-letters}(3) implies that $xy \in \Lin_{2d_1}(S)$.

It follows from Corollary~\ref{corollary:new-elements} and Lemma~\ref{lemma:descendingly-alternative-classes}(1) that $\Lin_{2d_1 + 1}(S) = \Lin_{2d_1}(S) + \Lin(S^{(2d_1 + 1)}) = \Lin_{2d_1}(S)$. By Lemma~\ref{lemma:full-stabilization}, in this case we have $\Lin_{\infty}(S) = \Lin_{2d_1}(S)$. Hence, by Definition~\ref{definition:system-length}, $l(S) \leq 2d_1$.
\end{proof}

We now prove an even stronger result.

\begin{lemma} \label{lemma:descendingly-alternative-length}
Let $S \subseteq \A$, $xy = (((y_1 x_2) \dots) x_k) \cdot (y_{n-k} (\dots (y_2 x_1))) \in S^n \setminus \Lin_{n-1}(S)$. Then $\dim \Lin_n(S) - d_0 \geq \max\{k,n-k\} + 2^n - 2^k - 2^{n-k} + 1$.
\end{lemma}

\begin{proof}
By Lemma~\ref{lemma:descendingly-alternative-classes}, $x_1, x_2, \dots, x_k$ are pairwise swappable, and $y_1, y_2, \dots, y_{n-k}$ are also pairwise swappable. Let now $l, m \in \mathbb{N}$, $1 \leq l \leq k$ and $1 \leq m \leq n-k$. We consider two arbitrary sequences of indices $1 \leq i_1 < \dots < i_l \leq k$ and $1 \leq j_1 < \dots < j_m \leq n-k$, and denote them by $I$ and $J$. We also denote $|I| = l$ and $|J| = m$. Then we set $w_{I,J} = (((y_{j_1} x_{i_2}) \dots) x_{i_l}) \cdot (y_{j_m} (\dots (y_{j_2} x_{i_1}))) \in S^{l+m}$. We will show that for any $2 \leq p \leq n$ the set
$$
W_p = \{ w_{I,J} \; | \; \text{all } I, J \text{ such that } |I| + |J| = p \}
$$
is linearly independent modulo $\Lin_{p-1}(S)$. It then follows that $d_p = d_p(S) = \dim \Lin_p(S) - \dim \Lin_{p-1}(S) \geq |W_p|$, so 
$$
d_2 + \dots + d_n \geq |W_2| + \dots + |W_n| = (2^k - 1)(2^{n-k} - 1) = 2^n - 2^k - 2^{n-k} + 1.
$$
Here $|W_2| + \dots + |W_n|$ is the cardinality of the set $\{ w_{I,J} \; | \; I, J \text{ are nonempty} \}$, and $2^k - 1$ and $2^{n-k} - 1$ are the numbers of all nonempty subsets in $\{ 1, \dots, k \}$ and $\{ 1, \dots, n - k \}$, respectively. Similarly to Corollary~\ref{corollary:descendingly-alternative-bound}, we may assume that $|S| = d_1$. Since $xy \notin \Lin_{n-1}(S)$, Proposition~\ref{proposition:alternating-letters}(3) then implies that we also have $d_1 \geq \max\{ k, n-k\}$. Therefore, 
$$
\dim \Lin_n(S) = d_0 + d_1 + d_2 + \dots + d_n \geq d_0 + \max\{ k, n-k \} + 2^n - 2^k - 2^{n-k} + 1,
$$
as desired.

Given $I$ and $J$, let us denote $\{ i_{l+1}, \dots, i_k \} = \{ 1, \dots, k \} \setminus I$ and $\{ j_{m+1}, \dots, j_{n-k} \} = \{ 1, \dots, n-k \} \setminus J$. Similarly to the proofs of Proposition~\ref{proposition:alternating-letters}(1) and Lemma~\ref{lemma:descendingly-alternative-classes}, one can show that
\begin{align*}
xy &{}= (((y_1 x_2) \dots) x_k) \cdot (y_{n-k} (\dots (y_2 x_1))) \\
&{}\sim \pm (((y_{j_1} x_{i_2}) \dots) x_{i_k}) \cdot (y_{j_{n-k}} (\dots (y_{j_2} x_{i_1}))) \\
&{}\sim \pm L_{y_{j_{n-k}}} \dots L_{y_{j_{m+1}}} ((((y_{j_1} x_{i_2}) \dots) x_{i_k}) \cdot (y_{j_{m}} (\dots (y_{j_2} x_{i_1})))) \\
&{}\sim \pm L_{y_{j_{n-k}}} \dots L_{y_{j_{m+1}}} R_{x_{i_k}} \dots R_{x_{i_{l+1}}} ((((y_{j_1} x_{i_2}) \dots) x_{i_l}) \cdot (y_{j_m} (\dots (y_{j_2} x_{i_1})))) \\
&{}= \pm L_{y_{j_{n-k}}} \dots L_{y_{j_{m+1}}} R_{x_{i_k}} \dots R_{x_{i_{l+1}}} w_{I,J}.
\end{align*}
Assume now that $I', J'$ are distinct from $I, J$ and $|I'| + |J'| = |I| + |J|$. We can apply the converse transformations to $u_{I',J'} = L_{y_{j_{n-k}}} \dots L_{y_{j_{m+1}}} R_{x_{i_k}} \dots R_{x_{i_{l+1}}} w_{I',J'}$ to show that $u_{I',J'}$ necessarily contains at least two equal letters which belong to the same class of pairwise swappable elements. Indeed, if $|I'| \geq |I|$ and $I' \neq I$, then there exists $i \in I' \setminus I$, so $x_i$ occurs in $u_{I',J'}$ twice: both in $w_{I',J'}$ and in $R_{x_i}$. Otherwise, we have $|J'| \geq |J|$ and $J' \neq J$, so some $y_j$ occurs both in $w_{I',J'}$ and in $L_{y_j}$. Hence, by Proposition~\ref{proposition:alternating-letters}(2), we have $u_{I',J'} \in \Lin_{n-1}(S)$.

Finally, assume that $W_p$ is linearly dependent modulo $\Lin_{p-1}(S)$. Then there exist $I, J$ such that $|I| + |J| = p$ and
$$
w_{I,J} = \sum_{\substack{|I'| + |J'| = p\\ (I',J') \neq (I,J)}} \alpha_{I',J'} w_{I',J'} + z,
$$
where $\alpha_{I',J'} \in \F$ and $z \in \Lin_{p-1}(S)$. Therefore, we have
\begin{align*}
xy &{}\sim \pm L_{y_{j_{n-k}}} \dots L_{y_{j_{m+1}}} R_{x_{i_k}} \dots R_{x_{i_{l+1}}} w_{I,J} \\
&{}= \pm \sum_{\substack{|I'| + |J'| = p\\ (I',J') \neq (I,J)}} \alpha_{I',J'} u_{I',J'} \pm L_{y_{j_{n-k}}} \dots L_{y_{j_{m+1}}} R_{x_{i_k}} \dots R_{x_{i_{l+1}}} z \in \Lin_{n-1}(S),
\end{align*}
so $xy \in \Lin_{n-1}(S)$, a contradiction.
\end{proof}

\begin{theorem} \label{theorem:descendingly-alternative-length}
Let $n = l(\A) \geq 2$. Then $\dim \A - d_0 \geq 2^{n-1} + n - 2$.
\end{theorem}

\begin{proof}
Choose a generating system $S$ for $\A$ such that $l(S) = l(\A)$. Then 
$$
\A = \Lin_n(S) \neq \Lin_{n-1}(S),
$$
so, by Corollary~\ref{corollary:new-elements} and Lemma~\ref{lemma:descendingly-alternative-classes}(1), there exists 
$$
xy = (((x_1 x_2) \dots) x_k) \cdot (y_{n-k} (\dots (y_2 y_1))) \in S^n \setminus \Lin_{n-1}(S).
$$
We may assume without loss of generality that $k \leq n - k$, i.e., $1 \leq k \leq \lfloor \tfrac{n}{2} \rfloor$. By Lemma~\ref{lemma:descendingly-alternative-length}, we have $\dim \A - d_0 = \dim \Lin_n(S) - d_0 \geq 2^n - 2^k - 2^{n-k} + n - k + 1 = f(n,k)$. Now note that for $k \leq \lfloor \tfrac{n}{2} \rfloor - 1$ it holds that $n - k - 1 > k$, so
$$
    f(n,k+1) - f(n,k) = (2^k + 2^{n-k} + k) - (2^{k+1} + 2^{n-k-1} + k + 1) = 2^{n-k-1} - 2^k - 1 \geq 2 - 1 > 0.
$$
Hence $f(n,k)$ is an increasing function in $k$ for $1 \leq k \leq \lfloor \tfrac{n}{2} \rfloor$, and thus its minimum is achieved when $k = 1$. We have $f(n,1) = 2^n - 2^{n-1} + n - 2 - 1 + 1 = 2^{n-1} + n - 2$, so $\dim \A - d_0 \geq 2^{n-1} + n - 2$.
\end{proof}

The following corollary provides a weaker but more conveninent form of Theorem~\ref{theorem:descendingly-alternative-length}.

\begin{corollary} \label{corollary:descendingly-alternative-length}
Let $\dim \A - d_0 \geq 3$. Then $l(\A) \leq \lceil \log_2(\dim \A - d_0) \rceil$.
\end{corollary}

\begin{proof}
We denote $n = \lceil \log_2(\dim \A - d_0) \rceil \geq 2$. Assume that $l(\A) \geq n + 1$. Then, by Theorem~\ref{theorem:descendingly-alternative-length}, we have
$$
2^n \geq \dim \A - d_0 \geq 2^n + (n + 1) - 2 \geq 2^n + 1,
$$
a contradiction. Therefore, $l(\A) \leq n$.
\end{proof}

\section{Lengths of descendingly flexible algebras} \label{section:sequences-flexible}

\subsection{Explicit form of words in descendingly flexible algebras}

In this section we assume that $\A$ is a descendingly flexible algebra over an arbitrary field~$\F$.

\begin{definition}
Let $S \subseteq \A$. We say that $x \in S^{(m)}$ is {\em left swapping} if 
$$
x = L_{x_m} R_{x_{m-1}} L_{x_{m-2}} \dots x_1
$$
for some $x_1, \dots, x_m \in S$, and $x$ is {\em right swapping} if 
$$
x = R_{x_m} L_{x_{m-1}} R_{x_{m-2}} \dots x_1.
$$
In this case $x_1$ is called {\em the inner letter} of $x$, and $x_2, \dots, x_m$ are called {\em the outer letters} of~$x$.
\end{definition}

\begin{lemma} \label{lemma:descendingly-flexible-bound}
Let $S \subseteq \A$ and $w \in S^{(m)}$ for some $m \geq 3$. Then
\begin{enumerate}[{\rm (1)}]
    \item The word $w$ is equivalent to $\pm (xy)z$ or $\pm z'(y'x')$, where $x, y', z'$ are left swapping, and $x', y, z$ are right swapping.
    \item The word $(xy)z$ or $z'(y'x')$ is obtained from $w$ by the permutation of letters, and this permutation preserves the classes of pairwise swappable letters (p.s.l.) in these words.
    \item The word $(xy)z$ or $z'(y'x')$ can be chosen in such a way that one of the following conditions holds (here $l(x)$ denotes the length of either $x$ or $x'$, and similarly for $l(y)$ and~$l(z)$): 
    \begin{itemize}
        \item[{\rm(EOO)}] $l(x)$ is even, $l(y)$, $l(z)$ are odd;
        \item[{\rm(OEE)}] $l(x)$ is odd, $l(y)$, $l(z)$ are even;
        \item[{\rm(O11)}] $l(x)$ is odd, $l(y) = l(z) = 1$;
        \item[{\rm(OO)}] $l(x) = 1$, $l(y)$ is even, $l(z)$ is odd, i.e., $w \sim \pm uv$, where $u = xy$ (or $u = z'$) is left swapping of odd length, $v = z$ (or $v = y'x'$) is right swapping of odd length, and at least one of the values $l(u)$ and $l(v)$ is greater than $1$;
        \item[{\rm(OE)}] $l(x)$, $l(y)$ are odd with $l(y) \geq 3$, $l(z) = 1$, i.e., either $w \sim \pm (xy)z \sim \mp (zy)x = \mp uv$ or, symmetrically, $w \sim \pm z'(y'x') \sim \mp x'(y'z') = \mp v'u'$. Here $u = zy$ is left swapping of even length with $l(u) \geq 4$, and $v = x$ is left swapping of odd length. Similarly, $u' = y'z'$ is right swapping of even length with $l(u') \geq 4$, and $v' = x'$ is right swapping of odd length.
    \end{itemize}
    Each of the first three types is coded by the sequence of symbols which describe the lengths of $l(x)$, $l(y)$ and $l(z)$ as {\em E} --- even, {\em O} --- odd, or $1$ --- equal to $1$. Following this notation, the last two types should be encoded as {\rm(1EO)} and {\rm(OO1)}, respectively. However, they are more easily described as a product of only two left or right swapping subwords~$u$ and~$v$, and Proposition~\ref{proposition:pulling-out} together with Lemma~\ref{lemma:descendingly-flexible-swapping-subwords} show that it is more convenient to operate with two subwords instead of three of them. Hence we use only two letters to encode the last two types.
    \item There are at most three classes of p.s.l. in $w$.
    
    Namely, in the first three cases of~(3) we can divide the letters of $(xy)z$ or $z'(y'x')$ into three types I, II and III such that all letters of each type are pairwise swappable, i.e., each type is completely contained in some class of p.s.l. In cases {\rm(EOO)} and {\rm(O11)} they are described by Table~\ref{table:types}(a), and in case {\rm(OEE)} we have Table~\ref{table:types}(b).

    In the last two cases it is sufficient to use only two types I and II, and both in {\rm(OO)} and {\rm(OE)} we have Table~\ref{table:types}(c).

    \begin{table}[H]
        \begin{minipage}{0.35\linewidth}
        \centering
        \begin{tabular}{|r|c|c|c|}
        \hline
        & $x/x'$ & $y/y'$ & $z/z'$ \\\hline
        outer & II/-- & III/-- & I/-- \\\hline
        inner & I & II & III \\\hline
        \end{tabular}
        \subcaption{First case}
        \end{minipage}
        \begin{minipage}{0.35\linewidth}
        \centering
        \begin{tabular}{|r|c|c|c|}
        \hline
        & $x/x'$ & $y/y'$ & $z/z'$ \\\hline
        outer & III/-- & I & II \\\hline
        inner & I & II & III \\\hline
        \end{tabular}
        \subcaption{Second case}
        \end{minipage}
        \begin{minipage}{0.27\linewidth}        
        \centering
        \begin{tabular}{|r|c|c|}
        \hline
        & $u/u'$ & $v/v'$ \\\hline
        outer & I/-- & II/-- \\\hline
        inner & II & I \\\hline
        \end{tabular}
        \subcaption{Third case}
        \end{minipage}
        \caption{Division of letters into types} \label{table:types}
        \end{table}
    Here $A/B$ denotes that one of the two possibilities $A$ or $B$ holds, and ``--'' means that the set of outer letters of the given word is empty.
    \item The largest class of p.s.l. contains at least $\lfloor \tfrac{m}{3} \rfloor + 1$ elements.
\end{enumerate}
\end{lemma}

\begin{proof}
\leavevmode
\begin{enumerate}[{\rm (1)}]
    \item We prove the statement by induction on $m$. The induction base is clearly true. In order to prove the induction step, it is sufficient to show that if $w \in S^{m}$ is an arbitrary word of the given form, then for any $s \in S$ the words $sw$ and $ws$ are equivalent to some words of the same form from $S^{m+1}$. We may assume without loss of generality that $w = (xy)z$, since the case $w = z'(y'x')$ is symmetric to it. Then, by Eqs.~\eqref{equation:general-equivalence-1} and~\eqref{equation:general-equivalence-2}, we have $ws = ((xy)z)s \sim -(sz)(xy)$ and $sw = s((xy)z) \sim -s((zy)x) \sim x((zy)s) \sim - x((sy)z)$. Since $sz$ and $sy$ are left alternative, these words have the desired form.

    Note that, when a new letter is added to $z$, the subwords $x$ and $y$ change their roles, and $z$ keeps its role. Similarly, when a new letter is added to $y$, the subwords $x$ and $z$ change their roles, and $y$ keeps its role. Normally, we can add letters only to $y$ or to $z$.
    \item This statement immediately follows from (1) and Definition~\ref{definition:swappable}.
    \item By~(1), we may assume without loss of generality that $w \sim \pm (xy)z$. We denote $x = L_{x_1} R_{x_2} L_{x_3} \dots x_j$, $y = R_{y_1} L_{y_2} R_{y_3} \dots y_k$ and $z = R_{z_1} L_{z_2} R_{z_3} \dots z_l$. If $k = l(y) \geq 3$, then $y = (y_2 y')y_1$, where $y' = R_{y_3} \dots y_k$ is again a right swapping word, and we can pull out the letters $y_1$ and $y_2$ from $w$ in the following way:
    $$
    w \sim \pm (xy)z = \pm (x((y_2 y') y_1)) z \sim \mp (x((y_2 y') z))y_1 \sim \pm (y_2((x y') z))y_1.
    $$
    We can proceed similarly with $(xy')z$ until we obtain $(xy'')z$ with $l(y'') \in \{ 1, 2 \}$. Similarly, if $l(z) \geq 3$, we can pull out the outer letters of $z$ from $w$ until we have $(xy'')z''$ with $l(z'') \in \{ 1, 2 \}$. We now consider four cases:
    \begin{itemize}
        \item $l(y'') = l(z'') = 1$. If $l(x)$ is even, then we have the case (EOO). The other case is temporarily denoted by (OOO).
        \item $l(y'') = 1$, $l(z'') = 2$. Then we have 
        $$
        w \sim \pm (xy'')(z_l z_{l-1}) \sim \mp z_{l-1}(z_l (xy'')).
        $$
        It remains to apply the previous case for $xy''$, $z_l$ and $z_{l-1}$.
        \item $l(y'') = 2$, $l(z'') = 1$. If $l(x) = 1$, then we have the case (OO). Otherwise, we can decompose $x = x_1 x'$, where $x'$ is right swapping, and obtain 
        $$
        w \sim \pm ((x_1x')y'')z'' \sim \mp (z'' y'')(x_1 x').
        $$
        Then $z'' y''$ is left swapping of length $3$, so we can apply the first case to $x'$, $x_1$ and $z''y''$.
        \item $l(y'') = l(z'') = 2$. Then 
        $$
        w \sim \pm (xy'')(z_l z_{l-1}) \sim \mp ((z_l z_{l-1}) y'')x \sim \pm ((y'' z_{l-1}) z_l)x \sim (x z_l)(y'' z_{l-1}).
        $$
        Then $y'' z_{l-1}$ is right swapping of length $3$, so we can apply the first case to $x$, $z_l$ and $y'' z_{l-1}$.
    \end{itemize}
    Then we push the letters from $y$ and $z$ back in, and after this procedure we are still in one of the cases (EOO), (OOO) or (OO), since it preserves the parity of $l(y)$ and $l(z)$. We now show that (OOO) leads to one of the cases (OEE), (O11) or (OE).

    If $l(y) = l(z) = 1$, then we already have the case (O11). Assume now the contrary. If $l(y) \geq 2$, then we can pull out one letter from $y$, and if $l(z) \geq 2$, we can do the same with $z$. In both cases we obtain a word with an even number of letters, since the parity of the length has changed. We have already shown that this word belongs to one of the types (EOO), (OOO) or (OO), but the type (OOO) has an odd number of letters, so it is impossible. Then we push the letter from $y$ or $z$ back in, and after that (EOO) is transformed into (OEE), while (OO) is transformed into (OE).
    
    \item We may assume without loss of generality that $w \sim \pm (xy)z$. We first consider the case (OEE). Then, proceeding as in the proof of~(3), we can pull out the outer letters of $y$ and $z$ until we obtain $(xy'')z''$ with $l(y'') = l(z'') = 2$. Due to transitivity of the swappability relation, all outer letters of $y$ are pairwise swappable, so they belong to the same type as $y_{k-1}$. Similarly, all outer letters of $z$ are pairwise swappable, and they belong to the same type as $z_{l-1}$. It remains to prove that the letters of $(xy'')z''$ are divided into types according to Table~\ref{table:types}(b). We have already shown in the proof of~(3) that we can obtain a subword $y'' z_{l-1}$, so $y_k \leftrightarrow z_{l-1}$, and a subword $x z_l$, so $x_1 \leftrightarrow z_l$ whenever $l(x) \geq 2$. Besides, we can pull out the outer letters of $x$ from a subword $x y''$ until we obtain $x_j y''$, and thus $x_j \leftrightarrow y_{k-1}$.

    If we have the case (O11) or $l(y) = l(z) = 1$ in (EOO), then the sets of outer letters for $y$ and $z$ are empty, so we immediately obtain Table~\ref{table:types}(a). Assume now that $l(y) \geq 3$ in (EOO). Then $y = y' y_1$, so we can pull $y_1$ out and obtain $w \sim \pm (x(y' y_1))z \sim \mp (x(y' z))y_1$. Then $x(y'z)$ is of the type (OEE), so it satisfies Table~\ref{table:types}(b) with the roles of $x$ and $z$ interchanged. All outer letters of $y$ are pairwise swappable, so $y_1$ belongs to the same type as the set of outer letters of $y'$ which is nonempty because $l(y') \geq 2$. Therefore, $(xy)z$ satisfies Table~\ref{table:types}(a). The case when $l(z) \geq 3$ is considered similarly.

    We now consider the case (OO). Then $w \sim \pm (xy)z = \pm uv$, where $u = xy$ and $v = z$. Clearly, all outer letters of $u$ are pairwise swappable, and, similarly to the proof of~(3), we can pull out the outer letters of $v = z$ until we obtain $l(z'') = 1$. Then we have $(xy)z'' = uz_l$, so the inner letter of $v$ is swappable with the outer letters of $u$. If $l(v) = 1$, then we have already shown that $uv$ satisfies Table~\ref{table:types}(c), and if $l(v) \geq 3$, then we can change the roles of $u$ and $v$ to show that the inner letter of $u$ is swappable with the outer letters of $v$.

    Now only the case (OE) remains. We have $w \sim \pm (xy)z \sim \mp uv$ with $u = zy$ and $v = x$. Clearly, $xy$ is of the type (OO), so it satisfies Table~\ref{table:types}(c) with $x$ and $y$ instead of $u$ and~$v$. Since $l(y) \geq 3$, the letter $z$ is swappable with the set of outer letters of $y$ which is nonempty, and thus $uv$ also satisfies Table~\ref{table:types}(c).
    
    \item If $3 \centernot\mid m$, then this statement immediately follows from (4). Assume now that $3 \mid m$, and all classes of p.s.l. contain less than $\tfrac{m}{3} + 1$ elements, i.e., there are three of them, and each one contains exactly $p = \tfrac{m}{3}$ elements. If $p = 1$, then $m = 3$, so there are only two classes of p.s.l. in this case, a contradiction. If $p \geq 2$, then this is possible only in cases (EOO) and (OEE), but then we obtain a contradiction with the fact that the sizes of all three classes have the same parity. \qedhere
\end{enumerate}
\end{proof}

Lemma~\ref{lemma:descendingly-flexible-bound}(5) has an immediate corollary which provides a linear bound on the length of a descendingly flexible algebra.

\begin{corollary} \label{corollary:descendingly-flexible-bound}
Let $S \subseteq \A$. Then $\Lin_{3d_1(S)}(S) = \Lin_{3d_1(S)-1}(S)$, so $l(S) \leq 3d_1(S) - 1$.
\end{corollary}

\begin{proof}
Similarly to the proof of Corollary~\ref{corollary:descendingly-alternative-bound}, we again assume that $|S| = d_1(S) = d_1$. Consider an arbitrary word $w \in S^{(3d_1)}$. By Lemma~\ref{lemma:descendingly-flexible-bound}(5), we have at least $d_1 + 1$ pairwise swappable letters in $(xy)z$, so, by Proposition~\ref{proposition:alternating-letters}(3), $w \in \Lin_{3d_1-1}(S)$. Then Corollary~\ref{corollary:new-elements} implies that $\Lin_{3d_1}(S) = \Lin_{3d_1-1}(S)$.
\end{proof}

However, we aim to obtain a logarithmic bound on the length of~$\A$. Our method uses the procedure of pulling the letters out which is described by the following proposition.

\begin{proposition} \label{proposition:pulling-out}
The division of letters into types from Lemma~\ref{lemma:descendingly-flexible-bound}(4) is preserved under pulling a letter out in the following sense. Here we follow the encoding from Lemma~\ref{lemma:descendingly-flexible-bound}(3).
    \begin{itemize}
        \item[{\rm(EOO)}] Consider the case with $(xy)z$. If $l(y) \geq 2$, then $y = y' y_1$, so we can pull $y_1$ out and obtain $(xy)z = (x(y' y_1))z \sim -(x(y' z))y_1$. If $l(z) \geq 2$, then $z = z' z_1$, so we can pull $z_1$ out and obtain $(xy)z = (xy)(z' z_1) \sim - z_1 (z' (xy))$. In both cases the words $x(y' z)$ and $z'(xy)$ have length $m-1$ and are of the type {\rm (OEE)}. Therefore, the division of their letters into types is described by Table~\ref{table:types}(b) with the roles of the subwords interchanged, and it can be obtained from Table~\ref{table:types}(a) applied to $(xy)z$.
        \item[{\rm(OEE)}] This case is completely similar to the previous one, but $x(y' z)$ and $z'(xy)$ are of the type {\rm (EOO)}, and the division of their letters into types is described by Table~\ref{table:types}(a).
        \item[{\rm(O11)}] We have $l(y) = l(z) = 1$, so we cannot pull out the letters from $y$ and $z$.
        \item[{\rm(OO)}] If $l(u) \geq 1$, then $u = u_1 u'$, so we can pull $u_1$ out and obtain $uv = (u_1 u') v \sim - (v u') u_1$. Similarly, if $l(v) \geq 1$, then $v = v' v_1$, so we can pull $v_1$ out and obtain $uv = u(v' v_1) \sim - v_1(v' u)$. The words $v u'$ and $v' u$ both have length $m-1$ and are of the type {\rm (OE)}. Hence the division of their letters into types is described by Table~\ref{table:types}(c), and it can be obtained from Table~\ref{table:types}(c) applied to $uv$.
        \item[{\rm(OE)}] Consider the case with $uv$. If $l(u) \geq 1$, then $u = u_1 u'$, and we again obtain $uv \sim - (v u') u_1$. Then the word $vu'$ has length $m-1$ and is of the type {\rm (OO)}, so the division of its letters into types is again described by Table~\ref{table:types}(c).
    \end{itemize}
\end{proposition}

\begin{proof}
Immediately follows from Lemma~\ref{lemma:descendingly-flexible-bound}(3--4).
\end{proof}

\begin{corollary} \label{corollary:pulling-out}
The division of letters into types from Lemma~\ref{lemma:descendingly-flexible-bound}(4) is preserved under pulling the letters out in the sence of Proposition~\ref{proposition:pulling-out} which is performed an arbitrary number of times.
\end{corollary}

\begin{proof}
Follows from Proposition~\ref{proposition:pulling-out} by induction on the number of iterations, since the cases (EOO) and (OEE), as well as (OO) and (OE), are transformed into each other under this procedure.
\end{proof}

\begin{remark} \label{remark:pulling-out}
In Proposition~\ref{proposition:pulling-out} it is possible to pull out not only one of the letters $y_1$, $z_1$, $u_1$ or $v_1$ (we will temporarily denote it by $s_1$, and the subword containing it by $s$). Actually, we can pull out an arbitrary letter~$t$ which belongs to the same type as $s_1$ in Table~\ref{table:types} (we will denote this type by IV). Indeed, $t$ and $s_1$ are swappable, so we can exchange $t$ with $s_1$ as in the proofs of Proposition~\ref{proposition:alternating-letters}(1) and Lemmas~\ref{lemma:descendingly-alternative-classes} and~\ref{lemma:descendingly-alternative-length}, so that $t$ becomes the most outer letter of $s$, and then pull $t$ out.

We can proceed to pull out the letters of type~IV until there is only one letter of this type is left, and it is the inner letter of some subword distinct from $s$. By Lemma~\ref{lemma:descendingly-flexible-bound}(4), type~IV contains $l(s)$ letters, so we can pull out at most $l(s) - 1$ of them.
\end{remark}

\begin{notation}
We denote the operation of pulling out several letters from $s$ as follows:
\begin{itemize}
    \item $\bm{E}s$ if their number is even;
    \item $\bm{O}s$ if their number is odd;
    \item $\bm{A}s$ if their number is arbitrary, i.e., either even or odd.    
\end{itemize}
\end{notation}

\begin{example} \label{example:pulling-out}
Consider $(xy)z \in S^n$ of the type {\rm (EOO)} or {\rm (OEE)} from Lemma~\ref{lemma:descendingly-flexible-bound}(3). If we pull out an even number of letters from $y$ ($\bm{E}y$), then the new word is again of the form $(xy)z$ (we denote this by $(xy)z \mapsto (xy)z$). If we pull out an odd number of letters from $z$ after that ($\bm{O}z$), then the word transforms into $z(xy)$ (i.e., $(xy)z \mapsto z(xy)$). If, finally, we pull out an arbitrary number of letters from $x$ ($\bm{A}x$), then for $\bm{E}x$ we have $z(xy) \mapsto z(xy)$, and for $\bm{O}x$ we have $z(xy) \mapsto (zx)y$.
\end{example}

\subsection{Bounds on the lengths}

We first prove an analogue of Lemma~\ref{lemma:descendingly-alternative-length}.

\begin{lemma} \label{lemma:descendingly-flexible-length}
Let $S \subseteq \A$, and $(xy)z \in S^n \setminus \Lin_{n-1}(S)$ is of the type {\rm (EOO)} or {\rm (OEE)} from Lemma~\ref{lemma:descendingly-flexible-bound}(3). Consider four sets of subwords obtained from $(xy)z$ by pulling its letters out:
\begin{enumerate}[{\rm (A)}]
    \item $\bm{E}y$ ($(xy)z \mapsto (xy)z$), $\bm{O}z$ ($(xy)z \mapsto z(xy)$), $\bm{A}x$;
    \item $\bm{E}z$ ($(xy)z \mapsto (xy)z$), $\bm{O}y$ ($(xy)z \mapsto x(yz)$), $\bm{A}x$;
    \item $\bm{O}z$ ($(xy)z \mapsto z(xy)$), $\bm{O}x$ ($z(xy) \mapsto (zx)y$), $\bm{A}y$;
    \item $\bm{O}y$ ($(xy)z \mapsto x(yz)$), $\bm{O}x$ ($x(yz) \mapsto (yz)x$), $\bm{A}z$.
\end{enumerate}
We denote by $W_m^X$ the subset of $X \in \{ A, B, C, D \}$ which consists of all words with a given length~$m$, $3 \leq m \leq n-1$. Then $W_m^X$ is linearly independent modulo $\Lin_{m-1}(S)$.
\end{lemma}

\begin{remark}
It is allowed to pull out not only the outer letters of $x$, $y$ and $z$ but all letters which belong to the same type (I, II or III). However, there should remain at least one letter of each type, see Remark~\ref{remark:pulling-out}.

By Corollary~\ref{corollary:pulling-out}, the subwords of length $m$ obtained by pulling out the same sets of letters have the same structure and the same division of letters into types. Thus they are equivalent up to multiplication up to $\pm 1$, so we will not distinguish them.
\end{remark}

\begin{proof}[Proof of Lemma~\ref{lemma:descendingly-flexible-length}]
We will prove this statement for the set A only, since other cases are completely similar. For simplicity of notation, we will also assume that the letters of $(xy)z$ are divided into types according to Table~\ref{table:types}(b). We denote the types containing the outer letters of $x$, $y$ and $z$ by $X = \{ x_1, \dots, x_j \}$, $Y = \{ y_1, \dots, y_k \}$ and $Z = \{ z_1, \dots, z_l \}$, respectively. Here $j = l(x)$, $k = l(y)$ and $l = l(z)$, so $j + k + l = n$. Note that the inner letter of $x$ belongs to $Y$, the inner letter of $y$ belongs to $Z$, and the inner letter of $z$ belongs to $X$. The elements of each type are pairwise swappable.

Let $m \in \mathbb{N}$. We consider three arbitrary sequences of indices $1 \leq j_1 < \dots < j_r \leq j$, $1 \leq k_1 < \dots < k_s \leq k$ and $1 \leq l_1 < \dots < l_t \leq l$, and denote them by $J$, $K$ and $L$, respectively. We also say that $k-s$ must be even, $l-t$ must be odd, and $r + s + t = m$. Then we define the word $w^A_{J,K,L} \in S^m$ as
$$
w^A_{J,K,L} = 
\begin{cases}
(L_{z_{l_1}} R_{z_{l_2}} \dots x_{j_r}) \cdot ((L_{x_{j_1}} R_{x_{j_2}} \dots y_{k_s}) \cdot (R_{y_{k_1}} L_{y_{k_2}} \dots z_{l_t})), & j-r \text{ is even},\\
((L_{z_{l_1}} R_{z_{l_2}} \dots x_{j_r}) \cdot (R_{x_{j_1}} L_{x_{j_2}} \dots y_{k_s})) \cdot (R_{y_{k_1}} L_{y_{k_2}} \dots z_{l_t}), & j-r \text{ is odd}.
\end{cases}
$$
This corresponds to Example~\ref{example:pulling-out} which describes the form of the word obtained from $(xy)z$ by pulling $\bm{E}y, \bm{O}z, \bm{A}x$ out. Clearly, $j-r$ is even if and only if $n-m$ is odd. We will show that for any $3 \leq m \leq n$ the set
$$
W_m^A = \{ w^A_{J,K,L} \; | \; \text{all } J, K, L \text{ such that } |J| + |K| + |L| = m \}
$$
is linearly independent modulo $\Lin_{m-1}(S)$. It will follow that $d_m(S) = \dim \Lin_m(S) - \dim \Lin_{m-1}(S) \geq |W_m^A|$.

Let us denote $\overline{q} = \{ 1, \dots, q \}$ for any $q \in \mathbb{N}$, and also $\{ j_{r+1}, \dots, j_j \} = \overline{j} \setminus J$, $\{ k_{s+1}, \dots, k_k \} = \overline{k} \setminus K$, $\{ l_{t+1}, \dots, l_l \} = \overline{l} \setminus L$. Then, by Corollary~\ref{corollary:pulling-out} and Remark~\ref{remark:pulling-out}, if we pull out $y_{k_{s+1}}, \dots, y_{k_k}$, then $z_{l_{t+1}}, \dots, z_{l_l}$, and finally $x_{j_{r+1}}, \dots, x_{j_j}$ from $(xy)z$, the resulting word is equivalent to $w^A_{J,K,L}$ up to multiplication by $\pm 1$.

Assume now that $J', K', L'$ are distinct from $J, K, L$ and $|J'| + |K'| + |L'| = |J| + |K| + |L|$. Then we can push the letters $x_{j_{r+1}}, \dots, x_{j_j}$, $z_{l_{t+1}}, \dots, z_{l_l}$, and $y_{k_{s+1}}, \dots, y_{k_k}$ into $w^A_{J',K',L'}$ to obtain a new word $u_{J',K',L'}$. It follows from Lemma~\ref{lemma:descendingly-flexible-bound}(4) and Corollary~\ref{corollary:pulling-out} that the letters of $u_{J',K',L'}$ are divided into types $J' \cup (\overline{j} \setminus J)$, $K' \cup (\overline{k} \setminus K)$ and $L' \cup (\overline{l} \setminus L)$. We now show that at least one of these types contains two equal elements. Indeed, if $|J'| \geq |J|$ and $J' \neq J$, then there are two equal elements in $J' \cup (\overline{j} \setminus J)$. The similar statement holds for $K$ and $L$, and at least one of these conditions must be satisfied, since otherwise we have $(J, K, L) = (J', K', L')$. Hence $u_{J',K',L'}$ has at least two equal letters which belong to the same class of p.s.l. and, by Proposition~\ref{proposition:alternating-letters}(2), we have $u_{J',K',L'} \in \Lin_{m-1}(S)$.

Finally, assume that $W_m^A$ is linearly dependent modulo $\Lin_{m-1}(S)$. Then there exist $J, K, L$ such that $|J| + |K| + |L| = m$ and
$$
w^A_{J,K,L} = \sum_{\substack{|J'| + |K'| + |L'| = m\\ (J',K',L') \neq (J,K,L)}} \alpha_{J',K',L'} w^A_{J',K',L'} + v,
$$
where $\alpha_{J',K',L'} \in \F$ and $v \in \Lin_{m-1}(S)$. Therefore, we have
$$
(xy)z \sim \pm \sum_{\substack{|J'| + |K'| + |L'| = m\\ (J',K',L') \neq (J,K,L)}} \alpha_{J',K',L'} u_{J',K',L'} \pm v',
$$
where $v'$ is obtained from $v$ by pushing the letters $x_{j_{r+1}}, \dots, x_{j_j}$, $z_{l_{t+1}}, \dots, z_{l_l}$, and $y_{k_{s+1}}, \dots, y_{k_k}$ in, so $v' \in \Lin_{n-1}(S)$. Hence $(xy)z \in \Lin_{n-1}(S)$, a contradiction.
\end{proof}

\begin{remark} \label{remark:descendingly-flexible-length}
Note that $W^A_{n-1}$ can only be obtained by pulling out one letter from $z$ ($\bm{1}z$), so $|W^A_{n-1}| = l$, $W^B_{n-1}$ can only be obtained by pulling out one letter from $y$ ($\bm{1}y$), so $|W^B_{n-1}| = k$, and $W^C_{n-1} = W^D_{n-1} = \varnothing$, since we have to pull out at least two letters from $(xy)z$. Besides, $W^A_{n-2} = W^C_{n-2}$ ($\bm{1}z, \bm{1}x$) and $W^B_{n-2} = W^D_{n-2}$ ($\bm{1}y, \bm{1}x$).
\end{remark}

\begin{lemma} \label{lemma:pulling-one-letter-subwords}
Assume that, under the conditions of Lemma~\ref{lemma:descendingly-flexible-length}, the set $W^A_{n-1} \cup W^B_{n-1}$ is linearly dependent modulo $\Lin_{n-2}(S)$. Then there exist two words $(\tilde{x}\tilde{y})\tilde{z}, (\hat{x}\hat{y})\hat{z} \in S^n \setminus \Lin_{n-1}(S)$ such that 
\begin{align*}
&\begin{cases}
l(\tilde{x}) = l(y) - 1,\\
l(\tilde{y}) = l(z),\\
l(\tilde{z}) = l(x) + 1.
\end{cases}
&
&\begin{cases}
l(\hat{x}) = l(z) - 1,\\
l(\hat{y}) = l(x) + 1,\\
l(\hat{z}) = l(y).
\end{cases}
\end{align*}
\end{lemma}

\begin{proof}
Assume that the set $W^A_{n-1} \cup W^B_{n-1}$ is linearly dependent. Since $W^A_{n-1}$ and $W^B_{n-1}$ are both linearly independent, it follows that there exists some $w^A_{\overline{j},\overline{k},L}$ with $|L| = l - 1$ which can be expressed linearly through 
$$
\left\{ w^A_{\overline{j},\overline{k},L'} \; \Big| \; |L'| = l - 1, \: L' \neq L \right\} \cup \left\{ w^B_{\overline{j},K',\overline{l}} \; \Big| \; |K'| = k - 1 \right\} \cup \Lin_{n-2}(S).
$$
Then
$$
(xy)z \sim \pm z_{l_l} w^A_{\overline{j},\overline{k},L} \sim \sum_{\substack{|L'| = l-1\\ L' \neq L}} \alpha_{L'} \left( z_{l_l} w^A_{\overline{j},\overline{k},L'} \right) + \sum_{|K'| = k-1} \beta_{K'} \left( z_{l_l} w^B_{\overline{j},K',\overline{l}} \right).
$$

We have shown in the proof of Lemma~\ref{lemma:descendingly-flexible-length} that $z_{l_l} w^A_{\overline{j},\overline{k},L'} \in \Lin_{n-1}(S)$. But $(xy)z \notin \Lin_{n-1}(S)$, so there exists $K'$ such that $z_{l_l} w^B_{\overline{j},K',\overline{l}} \notin \Lin_{n-1}(S)$. It can be easily seen that $w^B_{\overline{j},K',\overline{l}} \notin \Lin_{n-1}(S)$ has the form $x(y'z)$ with $l(y') = l(y) - 1$, so $z_{l_l} w^B_{\overline{j},K',\overline{l}}$ has the form $z_{l_l}(x(y'z)) \sim - (y'z)(x z_{l_l})$. It remains to denote $\tilde{x} = y'$, $\tilde{y} = z$ and $\tilde{z} = x z_{l_l}$.

Similarly, some $w^B_{\overline{j},K,\overline{l}}$ can be expressed linearly through all other elements of $W^A_{n-1} \cup W^B_{n-1}$ modulo $\Lin_{n-2}(S)$. By repeating this argument, we obtain the desired word $(\hat{x}\hat{y})\hat{z} \in S^n \setminus \Lin_{n-1}(S)$.
\end{proof}

Lemma~\ref{lemma:descendingly-flexible-length} does not allow to obtain subwords of length $m \leq 2$, since at least one letter of each of the three types must be left. Hence it provides lower bounds on $d_m(S)$ for $3 \leq m \leq n - 1$ only. In the following lemma we also obtain lower bounds on $d_1(S)$ and $d_2(S)$.

\begin{lemma} \label{lemma:descendingly-flexible-short-length}
Let $S \subseteq \A$, and $(xy)z \in S^n \setminus \Lin_{n-1}(S)$ is of the type {\rm (EOO)}, {\rm (OEE)} or {\rm (O11)} from Lemma~\ref{lemma:descendingly-flexible-bound}(3). We denote $j = l(x)$, $k = l(y)$ and $l = l(z)$. Then $d_1(S) \geq \max \{ j, k, l \}$ and $d_2(S) \geq \max \{ jk, kl, jl \}$.
\end{lemma}

\begin{proof}
By Lemma~\ref{lemma:descendingly-flexible-bound}(4), the letters of $(xy)z$ can be divided into three types whose sizes are $j$, $k$ and $l$, and the elements of each type are pairwise swappable. Similarly to the proof of Corollary~\ref{corollary:descendingly-alternative-bound}, we may assume that $|S| = d_1(S) = d_1$. Then it follows from Proposition~\ref{proposition:alternating-letters}(3) that $d_1(S) \geq \max \{ j, k, l \}$.

For simplicity of notation, in the proof of the second part of this lemma we may assume without loss of generality that $(xy)z$ is of the type (EOO), so its letters are divided into types according to Table~\ref{table:types}(a). We again denote the types containing the outer letters of $x$, $y$ and $z$ by $X = \{ x_1, \dots, x_j \}$, $Y = \{ y_1, \dots, y_k \}$ and $Z = \{ z_1, \dots, z_l \}$, respectively.

Let $j, k, l \geq 2$. Then $x = L_{x_1} R_{x_2} \dots z_l$, $y = R_{y_1} L_{y_2} \dots x_j$ and $z = R_{z_1} L_{z_2} \dots y_k$, so in the middle of $x$ we have $x_{j-1} z_l$, in the middle of $y$ we have $y_{k-1} x_j$, and in the middle of $z$ we have $z_{l-1} y_k$. We can now prove similarly to Lemma~\ref{lemma:descendingly-flexible-length} that the set 
$$
W_{x,z} = \{ x_h z_i \; | \; 1 \leq h \leq j, \, 1 \leq i \leq l \}
$$
is linearly independent. Indeed, if we replace $x_{j-1} z_l$ in $(xy)z$ with some other two-letter subword $x_h z_i$, then there occur two equal letters which are pairwise swappable, so the resulting word belongs to $\Lin_{n-1}(S)$. Hence there is no a linear expression for $x_{j-1} z_l$ through other elements of $W_{x,z}$ modulo $\Lin_1(S)$. The same is true for all other elements of the set $W_{x,z}$, so $W_{x,z}$ is linearly independent modulo $\Lin_1(S)$. Hence $d_2 = d_2(S) \geq |W_{x,z}| = jl$. Then we prove similarly that $d_2 \geq jk$ and $d_2 \geq kl$, so $d_2 \geq \max \{ jk, kl, jl \}$.

Assume now that one of the values $j, k, l$ is equal to $1$, say, $k = 1$. Then $\max \{ jk, kl, jl \} = jl$, and it is sufficient to show that $d_2 \geq jl$. If $j \geq 2$, then we again consider the set $W_{x,z}$ which is again linearly independent modulo $\Lin_{n-1}(S)$. If we also have $j = 1$, then we have to show that $d_2 \geq l$. If $l = 1$, then this is clearly true, since in this case $(xy)z \notin \Lin_2(S)$ implies $xy \notin \Lin_1(S)$. And if $l \geq 2$, then we consider the set 
$$
W_{z,y} = \{ z_h y_1 \; | \; 1 \leq h \leq l \}
$$
which is also linearly independent modulo $\Lin_{n-1}(S)$, and the statement follows.
\end{proof}

\begin{lemma} \label{lemma:descendingly-flexible-1-subwords}
Let $S \subseteq \A$, $n \geq 3$, and $(xy)z \in S^n \setminus \Lin_{n-1}(S)$ with $l(y) = l(z) = 1$. Then $\dim \Lin_n(S) - d_0 \geq 2^{n-2} + n - 2$.
\end{lemma}

\begin{proof}
Note that $(xy)z \notin \Lin_{n-1}(S)$ implies $xy \in S^{n-1} \setminus \Lin_{n-2}(S)$. Besides, $xy$ is a right swapping word, so we can denote $xy = R_{x_1} L_{x_2} \dots s$ for some $x_1, \dots, x_{n-2}, s \in S$. Similarly to Lemma~\ref{lemma:descendingly-flexible-length}, for any $1 \leq m \leq n-2$ and any sequence $1 \leq i_1 < \dots < i_m \leq n-2$ denoted by $I$ we can define an element $w_I \in S^{m+1}$ by
$$
w_I = \begin{cases}
R_{x_{i_1}} L_{x_{i_2}} \dots s, & n-m \text{ is even},\\
L_{x_{i_1}} R_{x_{i_2}} \dots s, & n-m \text{ is odd}.
\end{cases}
$$
Then the set
$$
W_m = \{ w_I \; | \; |I| = m \} \subseteq S^{m+1}
$$
is linearly independent modulo $\Lin_m(S)$, so $d_{m+1} = d_{m+1}(S) = \dim \Lin_{m+1}(S) - \dim \Lin_m(S) \geq |W_m|$. Besides, by Lemma~\ref{lemma:descendingly-flexible-short-length}, we have $d_1 \geq n - 2$. Thus
\begin{align*}
    \dim \Lin_n(S) - d_0 &{}= d_1 + (d_n + d_2 + \dots + d_{n-1}) \\
    &{}\geq n - 2 + (1 + |W_1| + \dots + |W_{n-2}|) = n - 2 + 2^{n-2}. \qedhere
\end{align*}
\end{proof}

\begin{lemma} \label{lemma:descendingly-flexible-low-length}
Let $S \subseteq \A$ and $d_n(S) \geq 1$ for some $1 \leq n \leq 5$, i.e., $l(S) \geq n$. Then
$$
\dim \Lin_n(S) - d_0 \geq
\begin{cases}
n, & 1 \leq n \leq 2,\\
2n - 1, & 3 \leq n \leq 5.
\end{cases}
$$
Moreover, if $n \in \{ 3, 4 \}$, then $d_k(S) \geq 2$ for any $1 \leq k \leq n-1$.
\end{lemma}

\begin{proof}
Recall that $\dim \Lin_n(S) - d_0 = d_1 + \dots + d_n$. By Proposition~\ref{proposition:full-stabilization}, if $d_m = 0$ for some $m \in \N$, then $d_k = 0$ for all $k \geq m$. Hence the statement is readily true for $n \leq 2$.

Let now $n = 3$. Then we may assume without loss of generality that there exists a word $(ab)c \in S^3 \setminus \Lin_2(S)$. We have $(ab)c \sim -(cb)a$, so the elements $ab, cb \in S^2$ are linearly independent modulo $\Lin_1(S)$, and $a, c \in S$ are linearly independent modulo $\Lin_0(S)$. Hence $d_3 \geq 1$ and $d_2, d_1 \geq 2$, so $\dim \Lin_3(S) - d_0 \geq 5$.

Assume that $n = 4$. If there exists some $(a(bc))d$ or $a((bc)d)\in S^4 \setminus \Lin_3(S)$, then there is a class of p.s.l. which consists of at least three elements, so one can easily show that $d_1, d_2, d_3 \geq 3$, and thus $\dim \Lin_4(S) - d_0 \geq 10$. Otherwise, there exists a word $(ab)(cd) \in S^4 \setminus \Lin_3(S)$, since $((cd)b)a \sim -(ab)(cd) \sim d(c(ab))$. Assume that $d_3 = 1$. Then $(cd)b, c(ab) \in S^3 \setminus \Lin_2(S)$ are linearly dependent modulo $\Lin_2(S)$. It then follows that we can replace $(cd)b$ with $c(ab)$ in $((cd)b)a \in S^4 \setminus \Lin_3(S)$ and obtain that $(c(ab))a \in S^4 \setminus \Lin_3(S)$. But we have already shown that in this case $d_3 \geq 3$, a contradiction. Therefore, $d_3 \geq 2$, and, according to the previous paragraph, we also have $d_2, d_1 \geq 2$. Thus $\dim \Lin_4(S) - d_0 \geq 7$.

Finally, let $n = 5$, and consider any $w \in S^5 \setminus \Lin_4(S)$. Then, by Lemma~\ref{lemma:descendingly-flexible-bound}(3), we may assume that $w$ belongs to one of the types (O11), (OEE) or (OE). In the case (O11), Lemma~\ref{lemma:descendingly-flexible-1-subwords} implies that $\dim \Lin_5(S) - d_0 \geq 2^3 + 3 = 11$. In the case (OEE), we assume that $w = (xy)z$ with $l(x) = 1$ and $l(y) = l(z) = 2$, so, by Lemma~\ref{lemma:descendingly-flexible-bound}(4), the outer letter of $z$ belongs to the type which consists of two elements. Thus we can pull it out from $(xy)z$ to obtain that $d_4 \geq 2$. In the case (OE), we may assume that $w = uv$ with $l(u) \in \{ 2, 4\}$, so, again by Lemma~\ref{lemma:descendingly-flexible-bound}(4), the type which contains the outer letters of $u$ consists of at least two elements. We can pull out an outer letter from $u$ and again obtain that $d_4 \geq 2$. Besides, we have already proved that $d_3, d_2, d_1 \geq 2$, so in the last two cases we have $\dim \Lin_5(S) - d_0 \geq 9$.
\end{proof}

\begin{lemma} \label{lemma:descendingly-flexible-swapping-subwords}
Let $S \subseteq \A$, $uv \in S^n \setminus \Lin_{n-1}(S)$, and one of the following conditions holds:
\begin{enumerate}[{\rm (1)}]
    \item both $u$ and $v$ are left swapping, $l(u) \geq 3$, either $l(u)$ is even or $l(v)$ is odd;
    \item both $u$ and $v$ are right swapping, $l(v) \geq 3$, either $l(v)$ is even or $l(u)$ is odd;
    \item $u$ is left swapping and $v$ is right swapping, at least one of the values $l(u)$ and $l(v)$ is odd, and if one of them equals $1$, then the other one is odd and greater than $1$.
\end{enumerate}
Then $\dim \Lin_n(S) - d_0 \geq 3 \cdot 2^{n-4} + n - 3$.
\end{lemma}

\begin{proof}
\leavevmode
\begin{enumerate}[{\rm (1)}]
\item If $l(u)$ is even and $l(v)$ is odd, then $uv$ has the type (OE), and the division of its letters into types is described by Table~\ref{table:types}(c). Hence type~I which contains the outer letters of~$u$ consists of $l(u)$ elements. By Proposition~\ref{proposition:pulling-out}, we can pull out an odd number of letters of type~I to obtain a word of the form $vu'$, where $u'$ is a subword of $u$ of odd length. Note that, since $l(u)$ is even, at least one letter of type~I always remains. After that we pull out an arbitrary number of outer letters from $v$. Then, similarly to the proof of Lemma~\ref{lemma:descendingly-flexible-length}, we can show that for each $m \in \mathbb{N}$ the resulting set $W_m$ of subwords of length $m$ is linearly independent modulo $\Lin_{m-1}(S)$, and thus $d_m \geq |W_m|$. Since there are $2^{l(u) - 1}$ ways to choose an odd number of letters of type~I, and there are $l(v) - 1$ outer letters in $v$, we obtain 
\begin{align*}
\dim \Lin_n(S) - d_0 &{}= d_1 + \dots + d_n \\
&{}\geq |W_1| + \dots + |W_n| \\
&{}= 2^{l(u) - 1} \cdot 2^{l(v) - 1} = 2^{n - 2} \\
&{}\geq 3 \cdot 2^{n-4} + n - 3    
\end{align*}
for any $n \geq 4$.

If $l(u)$ and $l(v)$ are both odd, then we set $u = z'$ and decompose $v = y'x'$ with $l(y') = 1$ and $x'$ right swapping of even length, so $uv = z'(y'x')$ has the type (EOO). Thus its letters are divided into types according to Table~\ref{table:types}(a), and hence the inner letter of~$v$ is swappable with outer letters of~$u$. Therefore, type~I which contains the outer letters of $u$ again consists of at least $l(u)$ elements. We can pull out an odd number of letters of type~I such that at least one letter of this type remains as the inner letter of $v$, and then pull out an arbitrary number of outer letters from $v$. If we denote the resulting sets of subwords of length $m$ by $W_m$, then we again obtain $d_m \geq |W_m|$. Note that $W_m$ is nonempty for $2 \leq m \leq n-1$ only, so we have bounds only for $d_2, \dots, d_{n-1}$. Since $l(u)$ is odd, we have to distract $1$ from $2^{l(u) - 1}$ ways to choose an even number of letters of type~I, and thus $|W_2| + \dots + |W_{n-1}| = (2^{l(u) - 1} - 1) \cdot 2^{l(v) - 1}$. We also have $d_n \geq 1$ and, by Lemma~\ref{lemma:descendingly-flexible-short-length}, $d_1 \geq \max\{ l(u), l(v) - 1 \}$. Therefore,
\begin{align*}
    \dim \Lin_n(S) - d_0 &{}= d_1 + \dots + d_n \\
    &{}\geq \max\{ l(u), l(v) - 1 \} + (2^{l(u) - 1} - 1) \cdot 2^{l(v) - 1} + 1 \\
    &{}= 2^{n-2} - 2^{n - l(u) - 1} + \max\{ l(u), n - l(u) - 1 \} + 1.
\end{align*}
The minimum of the right-hand side is achieved when $l(u)$ takes on its minimal value, i.e., $l(u) = 3$. Hence $\dim \Lin_n(S) - d_0 \geq 3 \cdot 2^{n-4} + n - 3$.

Finally, let $l(u)$ and $l(v)$ be both even. Then we denote $v = x$ and decompose $u = zy$, where $l(z) = 1$ and $y$ is right swapping of odd length. Thus $uv = (zy)x \sim -(xy)z$ is of the type (EOO), so, according to Table~\ref{table:types}(a), type~II, which contains the outer letters of $v = x$, also contains the inner letter of $y$, i.e., the inner letter of $u$. Hence type~II consists of $l(v)$ elements. Type~III, which contains the outer letters of $u$, consists of $l(u) - 1$ elements. We pull out an odd number of outer letters of $u$, and then we pull out an arbitrary number of letters  of type~II such that at least one of them remains as the inner letter of $u$. Similarly to the previous case, we denote the set of resulting subwords of length $m$ by $W_m$, and $d_m \geq |W_m|$ for any $2 \leq m \leq n-1$. We also have $d_n \geq 1$ and $d_1 \geq \max\{ l(u) - 1, l(v) \}$. Thus 
\begin{align*}
    \dim \Lin_n(S) - d_0 &= d_1 + \dots + d_n \\
    &\geq \max\{ l(u) - 1, l(v) \} + 2^{l(u) - 2} \cdot (2^{l(v)} - 1) + 1 \\
    &= 2^{n-2} - 2^{n - l(v) - 2} + \max\{ n - l(v) - 1, l(v) \} + 1.
\end{align*}
The minimum of the right-hand side is achieved when $l(v)$ takes on its minimal value, and since $l(v)$ is even, this is $l(v) = 2$. We again have $\dim \Lin_n(S) - d_0 \geq 3 \cdot 2^{n-4} + n - 3$.

\item This case is symmetric to the previous one.

\item By symmetricity, we may assume without loss of generality that either both $l(u)$ and $l(v)$ are odd, or $l(u)$ is odd and $l(v)$ is even.

In the first case $uv$ is of the type (OO), so, according to Table~\ref{table:types}(c), its elements are divided into two types of sizes $l(u)$ and $l(v)$. We can pull out an even number of letters of type~I to obtain a word of the form $u'v$, where $u'$ is a subword of odd length in $u$, and then pull out an even number of letters of type~II. Since $l(u)$ and $l(v)$ are odd, at least one letter of each type always remains. Thus $\dim \A - d_0 \geq 2^{l(u)-1} \cdot 2^{l(v) - 1} = 2^{n-2} \geq 3 \cdot 2^{n-4} + n - 3$ for any $n \geq 3$.

In the second case $l(u) \geq 3$, so we denote $v = z$ and decompose $u = xy$, where $l(x) = 1$ and $y$ is right swapping of even length. Then $uv = (xy)z$ has the type (OEE), so the division of its letters into types is described by Table~\ref{table:types}(b). Thus type~II which contains outer letters of $v = z$ also contains the inner letter of $y$, i.e., the inner letter of $u$, so type~II consists of $l(v)$ elements. We pull out an even number of outer letters of $u$, and then pull out an arbitrary number of letters of type~II such that at least one of them remains as the inner letter of $u$. Then $|W_2| + \dots + |W_{n-1}| = 2^{l(u) - 2} \cdot (2^{l(v)} - 1)$, $d_m \geq |W_m|$ for all $2 \leq m \leq n-1$, $d_n \geq 1$ and $d_1 \geq \max\{ l(u) - 1, l(v) \}$. Hence the bound for $\dim \A - d_0$ is the same as in the last paragraph of case~(1). \qedhere
\end{enumerate}
\end{proof}

\begin{lemma} \label{lemma:descendingly-flexible-high-length}
Let $S \subseteq \A$ and $d_n(S) \geq 1$ for some $n \geq 6$, i.e., $l(S) \geq n$. Then $\dim \Lin_n(S) - d_0 \geq 3 \cdot 2^{n-4} + n - 3$.
\end{lemma}

\begin{proof}
By Corollary~\ref{corollary:new-elements} and Lemma~\ref{lemma:descendingly-flexible-bound}, $d_n(S) \geq 1$ implies that there exists a word $w \in S^n \setminus \Lin_n(S)$ which has one of the types (EOO), (OEE), (O11), (OO) or (OE). If $w$ is of the type (O11), then the desired statement follows from Lemma~\ref{lemma:descendingly-flexible-1-subwords}, and if $w$ is of the type (OO) or (OE), then it follows from Lemma~\ref{lemma:descendingly-flexible-swapping-subwords}.

Assume now that $w = (xy)z$ is of the type (EOO). Then $n$ is even. If $l(y) = l(z) = 1$, then we again use Lemma~\ref{lemma:descendingly-flexible-1-subwords}. If $l(y) = 1$ and $l(z) \geq 3$, then we can consider $u = xy$ and $v = z$ which are both right swapping and of odd length, and then apply Lemma~\ref{lemma:descendingly-flexible-swapping-subwords}(2) to $uv$. If $l(y) \geq 3$ and $l(z) = 1$, then $w = (xy)z \sim - (zy)x$, so we consider $u = zy$ and $v = x$ which are both left swapping of even length, and then apply Lemma~\ref{lemma:descendingly-flexible-swapping-subwords}(1) to $uv$.

Let now $l(y), l(z) \geq 3$. We denote $l(x) = j$, $l(y) = k$ and $l(z) = l$. Consider the sets $W_m^X$ from Lemma~\ref{lemma:descendingly-flexible-length}. Then $d_3 + \dots + d_{n-1} \geq \Sigma^X$, where
$$
\Sigma^X = |W_3^X| + \dots + |W_{n-1}^X| =
\begin{cases}
    2^{k-1} \cdot (2^{l-1} - 1) \cdot (2^j - 1), & X = A,\\
    2^{l-1} \cdot (2^{k-1} - 1) \cdot (2^j - 1), & X = B,\\
    (2^{l-1} - 1) \cdot 2^{j-1} \cdot (2^k - 1), & X = C,\\
    (2^{k-1} - 1) \cdot 2^{j-1} \cdot (2^l - 1), & X = D.
\end{cases}
$$
Besides, $d_n \geq 1$ and, by Lemma~\ref{lemma:descendingly-flexible-short-length}, $d_1 \geq \max \{ j, k, l \}$ and $d_2 \geq \max \{ jk, kl, jl \}$. It also follows from Lemma~\ref{lemma:descendingly-flexible-length} applied to $W_{n-1}^A$ and $W_{n-1}^B$ that $d_{n-1} \geq \max \{ k, l \}$, and Remark~\ref{remark:descendingly-flexible-length} implies that $|W_{n-1}^C| = |W_{n-1}^D| = 0$. We may assume without loss of generality that $l \geq k$. Then $\Sigma^A \geq \Sigma^B$ and $\Sigma^C \geq \Sigma^D$. 
\begin{itemize}
    \item If, moreover, $k \geq j$, then $\Sigma^C \geq \Sigma^A$, so in this case we can write
    \begin{align*}
    \dim \Lin_n(S) - d_0 &{}= d_1 + d_2 + (d_3 + \dots + d_{n-2}) + d_{n-1} + d_n\\
    &{}\geq l + kl + (2^{l-1} - 1) \cdot 2^{j-1} \cdot (2^k - 1) + l + 1\\
    &{}= 2^{n-2} + 2^{n-k-l-1} - 2^{n-l-1} - 2^{n-k-2} + l (k+2) + 1.
    \end{align*}
    Minimizing the right-hand side over all $k$ and $l$ such that $n \geq l \geq k \geq n - l - k \geq 2$ with $k$ and $l$ being odd, we obtain that the minimum is achieved for $j = 2$, $k = 3$ and $l = n-5$, which is possible for even values of $n \geq 8$ only. Hence
    \begin{align*}
    \dim \Lin_n(S) - d_0 &{}\geq 2^{n-2} + 2 - 2^4 - 2^{n-5} + 5(n-5) + 1 \\
    &{}= 7 \cdot 2^{n-5} + 5n - 38 \geq 3 \cdot 2^{n-4} + n - 3.
    \end{align*}
    \item Otherwise, we have $\Sigma^A \geq \Sigma^C$. Moreover, if the set $W^A_{n-1} \cup W^B_{n-1}$ is linearly dependent modulo $\Lin_{n-2}(S)$, then, by Lemma~\ref{lemma:pulling-one-letter-subwords}, there exists a word $(\tilde{x}\tilde{y})\tilde{z} \in S^n \setminus \Lin_{n-1}(S)$ with $l(\tilde{x}) = k-1$, $l(\tilde{y}) = l$ and $l(\tilde{z}) = j + 1$, so $l(\tilde{y}), l(\tilde{z}) \geq l(\tilde{x})$, and we fall into the previous case. Therefore, we may assume that $W^A_{n-1} \cup W^B_{n-1}$ is linearly independent modulo $\Lin_{n-2}(S)$, and thus $d_{n-1} \geq |W^A_{n-1}| + |W^B_{n-1}|$, so we can add $|W^B_{n-1}| = k$ to~$\Sigma^A$. Hence
    \begin{align*}
    \dim \Lin_n(S) - d_0 &{}= d_1 + d_2 + (d_3 + \dots + d_{n-1}) + d_n\\
    &{}\geq \max \{ j, l \} + jl + 2^{k-1} \cdot (2^{l-1} - 1) \cdot (2^j - 1) + k + 1\\
    &{}= 2^{n-2} + 2^{n-j-l-1} - 2^{n-l-1} - 2^{n-j-2} \\
    &{}+ \max \{ j, l \} + jl + (n - j - l) + 1.
    \end{align*}
    The minimum of the right-hand side over all $j$ and $l$ such that $n \geq j, l \geq n - j - l \geq 3$ with $j$ even and $l$ odd is achieved for $k = l = 3$ and $j = n - 6$. This is possible for $n \geq 10$ only. Thus
    \begin{align*}
    \dim \Lin_n(S) - d_0 &{}\geq 2^{n-2} + 2^2 - 2^{n-4} - 2^4 + (n-6) + 3(n-6) + 3 + 1 \\
    &{}= 3 \cdot 2^{n-4} + 4n - 32 \geq 3 \cdot 2^{n-4} + n - 3.
    \end{align*}
\end{itemize}

Assume now that $w = (xy)z$ is of the type (OEE). Then $n$ is odd. If $l(x) = 1$, then we consider a left swapping word of odd length $u = xy$ and a right swapping word of even length $v = z$, and then apply Lemma~\ref{lemma:descendingly-flexible-swapping-subwords}(3) to $uv$. Let now $l(x) \geq 3$. Then, in the notations of the previous case, we have 
$$
\Sigma^X = |W_3^X| + \dots + |W_{n-1}^X| =
\begin{cases}
    (2^{k-1} - 1) \cdot 2^{l-1} \cdot (2^j - 1), & X = A,\\
    (2^{l-1} - 1) \cdot 2^{k-1} \cdot (2^j - 1), & X = B,\\
    2^{l-1} \cdot (2^{j-1} - 1) \cdot (2^k - 1), & X = C,\\
    2^{k-1} \cdot (2^{j-1} - 1) \cdot (2^l - 1), & X = D.
\end{cases}
$$
We again assume without loss of generality that $l \geq k$. Then $\Sigma^A \leq \Sigma^B$ and $\Sigma^C \leq \Sigma^D$. 
\begin{itemize}
    \item If, moreover, $j \geq l$, then $\Sigma^D \geq \Sigma^B$, so in this case we can write
    \begin{align*}
    \dim \Lin_n(S) - d_0 &{}= d_1 + d_2 + (d_3 + \dots + d_{n-2}) + d_{n-1} + d_n\\
    &{}\geq j + jl + 2^{k-1} \cdot (2^{j-1} - 1) \cdot (2^l - 1) + l + 1\\
    &{}= 2^{n-2} + 2^{n-j-l-1} - 2^{n-j-1} - 2^{n-l-2} + (j + 1)(l + 1).
    \end{align*}
    Minimizing the right-hand side over all $j$ and $l$ such that $n \geq j \geq l \geq n - l - j \geq 2$ with $j$ odd and $l$ even, we obtain that the minimum is achieved for $j = n - 4$ and $k = l = 2$, which is possible for odd values of $n \geq 7$ only. Hence
    \begin{align*}
    \dim \Lin_n(S) - d_0 &{}\geq 2^{n-2} + 2 - 2^3 - 2^{n-4} + 3 (n - 3) \\
    &{}= 3 \cdot 2^{n-4} + 3n - 15 \geq 3 \cdot 2^{n-4} + n - 3.
    \end{align*}
    \item Otherwise, we have $\Sigma^B \geq \Sigma^D$, so
    \begin{align*}
    \dim \Lin_n(S) - d_0 &{}= d_1 + d_2 + (d_3 + \dots + d_{n-1}) + d_n\\
    &{}\geq l + \max \{ jl, kl \} + (2^{l-1} - 1) \cdot 2^{k-1} \cdot (2^j - 1) + 1\\
    &{}= 2^{n-2} + 2^{n-j-l-1} - 2^{n-l-1} - 2^{n-j-2} \\
    &{}+ l \cdot (\max \{ j, n-j-l \} + 1) + 1.
    \end{align*}
    The minimum of the right-hand side over all $j$ and $l$ such that $n \geq l \geq j, n - j - l \geq 2$ with $j$ odd and $l$ even is achieved for $j = 3$, $k = 2 \left\lfloor \frac{n-3}{4} \right\rfloor$ and $l = 2 \left\lceil \frac{n-3}{4} \right\rceil$. This is possible for $n \geq 9$ only. Thus $k \leq n - 7$, so 
    \begin{align*}
        \dim \Lin_n(S) - d_0 &{}\geq 2^{n-2} + 2^{k-1} - 2^{k+2} - 2^{n-5} + (3 + 1) \cdot l + 1 \\
        &{}\geq 7 \cdot (2^{n-5} - 2^{k-1}) + 4 \cdot 2\cdot \tfrac{n-3}{4} + 1 \\
        &{}\geq 7 \cdot (2^{n-5} - 2^{n-8}) + 2n - 5 = 49 \cdot 2^{n-8} + 2n - 5\\
        &{}\geq 48 \cdot 2^{n-8} + n - 3 = 3 \cdot 2^{n-4} + n - 3.
    \end{align*}
\end{itemize}

Hence the statement is proved in all cases.
\end{proof}

\begin{theorem} \label{theorem:descendingly-flexible-length}
Let $n = l(\A)$. Then 
$$
\dim \A - d_0 \geq
\begin{cases}
n, & 1 \leq n \leq 2,\\
2n - 1, & 3 \leq n \leq 5,\\
3 \cdot 2^{n-4} + n - 3, & 6 \leq n.
\end{cases}
$$
\end{theorem}

\begin{proof}
Choose a generating system $S$ for $\A$ such that $l(S) = l(\A)$. Then the statement immediately follows from Lemmas~\ref{lemma:descendingly-flexible-low-length} and~\ref{lemma:descendingly-flexible-high-length}.
\end{proof}

\begin{corollary} \label{corollary:descendingly-flexible-length}
It holds that
$$
l(\A) \leq
\begin{cases}
    \left\lceil \frac{\dim \A - d_0}{2} \right\rceil, & 3 \leq \dim \A - d_0 \leq 10,\\
    \lceil \log_2(\dim \A - d_0) + \log_2(8/3) \rceil, & 11 \leq \dim \A - d_0.
\end{cases}
$$
\end{corollary}

\begin{proof}
The case when $3 \leq \dim \A - d_0 \leq 10$ immediately follows from Theorem~\ref{theorem:descendingly-flexible-length}. Let now $\dim \A - d_0 \geq 11$. We denote $n = \lceil \log_2(\dim \A - d_0) + \log_2(8/3) \rceil \geq 5$. Assume that $l(\A) \geq n + 1 \geq 6$. Then, by Theorem~\ref{theorem:descendingly-flexible-length}, we have
$$
3 \cdot 2^{n-3} \geq \dim \A - d_0 \geq 3 \cdot 2^{n-3} + (n + 1) - 3 \geq 3 \cdot 2^{n-3} + 3,
$$
a contradiction. Therefore, $l(\A) \leq n$.
\end{proof}


\begin{thebibliography}{99}
% \bibitem{Baez}
% Baez, J.~C. (2001). The octonions. {\em Bull. Amer. Math. Soc. (N.S.)} 39(2): 145--205. DOI: 10.1090/s0273-0979-01-00934-x
% \bibitem{Beites}
% Beites, P.~D., Nicol\'as, A.~P. (2017). A note on standard composition algebras of types~II and~III. {\em Adv. Appl. Clifford Algebr.} 27(2): 955--964. DOI: 10.1007/S00006-016-0668-8
% \bibitem{Benkart}
% Benkart, G.~M., Britten, D.~J., Osborn, J.~M. (1982). Real flexible division algebras. {\em Can. J. Math.} 34(3): 550-588. DOI: 10.4153/CJM-1982-039-X
% \bibitem{Bott}
% Bott, R., Milnor, J. (1958). On the parallelizability of the spheres. {\em Bull. Am. Math. Soc.} 64: 87--89. DOI: 10.1090/S0002-9904-1958-10166-4
% \bibitem{Dokovic}
% \DJ okovi\'c, D.~\v{Z}., Zhao, K. (2004). Real division algebras with large automorphism group. {\em J. Algebra.} 282(2): 758--796. DOI: 10.1016/J.JALGEBRA.2004.03.015
\bibitem{Elduque1}
Elduque,~A. (1997). Symmetric composition algebras. {\em J. Algebra} 196(1): 282--300. DOI: 10.1006/jabr.1997.7071
% \bibitem{Elduque2}
% Elduque, A. (2015). Okubo algebras: automorphisms, derivations and idempotents, in: {\em Lie algebras and related topics}, in: {\em Contemp. Math.} 652: 61--73. Providence, RI: Amer. Math. Soc. DOI: 10.1090/conm/652/12953
\bibitem{Elduque5}
Elduque,~A. (2018). Order $3$ elements in $G_2$ and idempotents in symmetric composition algebras. {\em Canad. J. Math.} 70(5): 1038--1075.
\bibitem{ElduqueMyung1}
Elduque, A., Myung, H.~Ch. (1991). Flexible composition algebras and Okubo algebras. {\em Comm. Algebra.} 19(4): 1197--1227. DOI: 10.1080/00927879108824198
% \bibitem{ElduqueMyung2}
% Elduque, A., Myung, H.~Ch. (1993). On flexible composition algebras. {\em Comm. Algebra.} 21(7): 2481--2505. DOI: 10.1080/00927879308824688
\bibitem{ElduqueMyung3}
Elduque, A., Myung, H.~Ch. (2004). Composition algebras satisfying the flexible law. {\em Comm. Algebra.} 32(5): 1997--2013. DOI: 10.1081/AGB-120029918
% \bibitem{ElduquePerez1}
% Elduque, A., P\'erez, J.~M. (1996). Composition algebras with associative bilinear form. {\em Comm. Algebra.} 24(3): 1091--1116. DOI: 10.1080/00927879608825625
% \bibitem{ElduquePerez2}
% Elduque, A., P\'erez, J.~M. (1997). Infinite dimensional quadratic forms admitting composition. {\em Proc. Amer. Math. Soc.} 125(8): 2207--2216. DOI: 10.1090/S0002-9939-97-03799-4
\bibitem{ElduquePerez3}
Elduque, A., P\'erez, J.~M. (1997). Composition algebras with large derivation algebras. {\em J. Algebra.} 190: 372--404. DOI: 10.1006/JABR.1996.6851
%\bibitem{Furey}
%Furey, C. (2015). Standard model physics from an algebra? {\em Ph. D. Theses, Waterloo, Ontario, Canada.} Available at: \texttt{arXiv:1611.09182} DOI: 10.17863/CAM.37465
\bibitem{Futorny}
Futorny, V., Horn, R.~A., Sergeichuk, V.~V. (2017). Specht's criterion for systems of linear mappings. {\em Linear Algebra Appl.} 519: 278--295. DOI: 10.1016/j.laa.2017.01.006
%\bibitem{Girard}
%Patrick R.~Girard, Patrick Clarysse, Romaric Pujol, Robert Goutte and Philippe Delachartre,  Hyperquaternions: A New Tool for Physics, Adv. Appl. Clifford Algebras (2018) 28:68, 1--14.
% \bibitem{Guterman_Hurwitz-algebras}
% Guterman, A.~E., Kudryavtsev, D.~K. (2017). The lengths of the quaternion and octonion algebras. {\em J. Math. Sci. (N.-Y.)} 224(6): 826--832. DOI: 10.1007/S10958-017-3453-X
\bibitem{Guterman_upper-bounds}
Guterman, A.~E., Kudryavtsev, D.~K. (2020). Upper bounds for the length of non-associative algebras. {\em J. Algebra.} 544: 483--497. DOI: 10.1016/j.jalgebra.2019.10.030
\bibitem{Guterman_sequences}
Guterman, A.~E., Kudryavtsev, D.~K. (2020). Characteristic sequences of non-associative algebras. {\em Comm. Algebra.} 48(4): 1713--1725. DOI: 10.1080/\-00927872.\-2019.1705469
\bibitem{Guterman_quadratic}
Guterman, A.~E., Kudryavtsev, D.~K. (2021). Length function and characteristic sequences of quadratic algebras. {\em J. Algebra.} 579: 428--455. DOI: 10.1016/j.jalgebra.2021.04.001
\bibitem{Guterman_slowly-growing}
Guterman, A.~E., Kudryavtsev, D.~K. (2022). Algebras of slowly growing length. {\em  Internat. J. Algebra Comput.} DOI: 10.1142/S0218196722500564
%\bibitem{Guterman_general-method}
%A.~E. Guterman, D.~K. Kudryavtsev, \textit{General method to compute the lengths of non-associative algebras.} --- Preprint.
% \bibitem{our_split-algebras}
% Guterman, A.~E., Zhilina, S.~A. (2023). Cayley-Dickson split-algebras: doubly alternative zero divisors and relation graphs. {\em J. Math. Sci.} 269(3): 331--355. DOI: 10.1007/s10958-023-06285-5
\bibitem{our_standard-composition-algebras}
Guterman, A.~E., Zhilina, S.~A. (2022). On the lengths of standard composition algebras. {\em Comm. Algebra.} 50(3): 1092--1105. DOI: 10.1080/00927872.2021.1977945
\bibitem{GZh_Okubo}
Guterman, A.~E., Zhilina, S.~A. (2023). On the lengths of Okubo algebras. Preprint.
% \bibitem{Jacobson}
% Jacobson, N. (1958). Composition algebras and their automorphisms. {\em Rend. Circ. Mat. Palermo.} 7: 55--80. DOI: 10.1007/978-1-4612-3694-8\_24
% \bibitem{Kervaire}
% Kervaire, M. (1958). Non-parallelizability of the $n$-sphere for $n > 7$. {\em Proc. Nat. Acad. Sci.} 44(3): 280--283. DOI: 10.1073/PNAS.44.3.280
% \bibitem{Knus}
% Knus, M.-A., Merkurjev, A., Rost, M., Tignol, J.-P. (1998). {\em The book of involutions.} AMS Colloquium Publications. 44. Providence, RI: American Mathematical Society.
% \bibitem{Laffey}
% Laffey,~Th., Markova,~O., Šmigoc,~H. (2016). The effect of assuming the identity as a generator on the length of the matrix algebra. {\em Linear Algebra Appl.} 498: 378--393. DOI: 10.1016/j.laa.2015.09.021
\bibitem{Markova}
Markova, O.~V. (2010). Matrix algebras and their length, in: {\em Matrix methods: theory, algorithms and applications.} Hackensack, NJ: World Sci. Publ., pp. 116--139. DOI: 10.1142/9789812836021\_0007
\bibitem{McCrimmon}
McCrimmon, K. (2004). {\em A Taste of Jordan Algebras.} New York, NY: Springer-Verlag. DOI: 10.1007/b97489
% \bibitem{Okubo1}
% Okubo, S. (1978). Pseudo-quaternion and pseudo-octonion algebras. {\em Hadronic J.} 1(4): 1250--1278.
% \bibitem{Okubo2}
% Okubo, S. (1978). Deformation of the Lie-admissible pseudo-octonion algebra into the octonion algebra. {\em Hadronic J.} 1(5): 1383--1431.
% \bibitem{Okubo_flexible1}
% Okubo, S. (1982). Classification of flexible composition algebras. I. {\em Hadronic J.} 5(4): 1564--1612.
% \bibitem{Okubo_flexible2}
% Okubo, S. (1982). Classification of flexible composition algebras. II. {\em Hadronic J.} 5(4): 1613--1626.
% \bibitem{Okubo}
% Okubo, S. (1995). {\em Introduction to octonion and other non-associative algebras in physics.} Cambridge, UK: Cambridge University Press. DOI: 10.1017/\-CBO978051\-1524479
% \bibitem{Okubo_symmetric1}
% Okubo, S., Osborn, J. M. (1981). Algebras with nondegenerate associative symmetric bilinear forms permitting composition. {\em Comm. Algebra.} 9(12): 1233--1261. DOI: 10.1080/00927878108822644
% \bibitem{Okubo_symmetric2}
% Okubo, S., Osborn, J. M. (1981). Algebras with nondegenerate associative symmetric bilinear forms permitting composition. II. {\em Comm. Algebra.} 9(20): 2015--2073. DOI: 10.1080/00927878108822695
\bibitem{Pappacena}
Pappacena, C.~J. (1997). An upper bound for the length of a finite-dimensional algebra. {\em J. Algebra.} 197(2): 535--545. DOI: 10.1006/jabr.1997.7140
\bibitem{Paz}
Paz, A. (1984). An application of the Cayley--Hamilton theorem to matrix polynomials in several variables. {\em Linear Multilin. Algebra.} 15(2): 161--170. DOI: 10.1080/\-030810884\-08817585
\bibitem{Pearcy}
Pearcy, C. (1962). A complete set of unitary invariants for operators generating finite $W^*$-algebras of type~I. {\em Pacific J. Math.} 12: 1405--1416. DOI: 10.2140/PJM.1962.12.1405
% \bibitem{Roos}
% Roos, G. (2008). Symmetries in complex analysis, in: {\em Exceptional symmetric domains. Contemp. Math.} 468: 157--189. Providence, RI: American Mathematical Society. DOI: 10.1090/CONM/468
\bibitem{Schafer}
Schafer, R. D. (1954). On the algebras formed by the Cayley-Dickson process. {\em Amer. J. Math.} 76(2): 435--446. DOI: 10.2307/2372583
\bibitem{Shitov}
Shitov, Y. (2019). An improved bound for the lengths of matrix algebras. {\em Algebra Number Theory.} 13(6): 1501--1507. DOI: 10.2140/ant.2019.13.1501
\bibitem{Specht}
Specht, W. (1940). Zur Theorie der Matrizen II. {\em  Jber. Deutsch. Math.-Verein.} 50: 19--23.
% \bibitem{Spencer-1959}
% Spencer, A.~J.~M., Rivlin, R.~S. (1958). The theory of matrix polynomials and its application to the mechanics of isotropic continua. {\em Arch. Ration. Mech. Anal.} 2: 309--336. DOI: 10.1007/BF00277933
% \bibitem{Spencer-1960}
% Spencer, A.~J.~M., Rivlin, R.~S. (1959). Further results in the theory of matrix polynomials. {\em Arch. Ration. Mech. Anal.} 4: 214--230. DOI: 10.1007/BF00281388
% \bibitem{Zhevlakov}
% Zhevlakov, K. A., Slin'ko, A. M., Shestakov, I. P., Shirshov, A. I. (1982). {\em Rings that are nearly associative.} New York, NY: Academic Press.
\end{thebibliography}
\end{document}